\documentclass[11pt,a4paper,reqno]{article}
\usepackage[utf8]{inputenc}
\usepackage[T1]{fontenc}
\usepackage[english]{babel}
\usepackage{nicefrac}

\usepackage{hyperref}
\usepackage{graphicx}
\usepackage{amsfonts,amssymb,amsbsy,latexsym,a4wide,xspace}
\usepackage{amsmath,delarray,enumerate,calc,amsthm}
\usepackage{bbm}
\usepackage{txfonts}
\usepackage{paralist}
\usepackage{xcolor}
\usepackage{authblk}

\usepackage{url}

\newcommand\blfootnote[1]{%
  \begingroup
  \renewcommand\thefootnote{}\footnote{#1}%
  \addtocounter{footnote}{-1}%
  \endgroup
}

\newcommand{\const}{{\textrm const}}

\newcommand{\id}{\operatorname{id}}
\newcommand{\supp}{\operatorname{supp}}

\newcommand{\cA}{{\mathcal A}}
\newcommand{\cC}{{\mathcal C}}
\newcommand{\cB}{{\mathcal B}}

\newcommand{\cF}{{\mathcal F}}

\newcommand{\cO}{{\mathcal O}}

\newcommand{\tAB}{\widetilde{\cA}_{\rm Borel}}
\newcommand{\tAb}{{\widetilde{\cA}_{\rm ball}}}

\newcommand{\cM}{{\mathcal M}}
\newcommand{\cN}{{\mathcal N}}

\newcommand{\R}{{\mathbbm R}}
\newcommand{\N}{{\mathbbm N}}
\newcommand{\Q}{{\mathbbm Q}}
\newcommand{\Z}{{\mathbbm Z}}
\newcommand{\C}{{\mathbbm C}}

\newcommand{\T}{{\mathbbm T}}

\newcommand{\lcm}{\operatorname{lcm}}
\newcommand{\card}{\#} 

\newcommand{\tS}{{\widetilde S}}
\newcommand{\tX}{{\widetilde X}}
\newcommand{\tu}{\tilde{u}}
\newcommand{\tx}{\tilde{x}}
\newcommand{\ty}{\tilde{y}}

\newcommand{\td}{{\tilde d}}

\newcommand{\tpi}{{\widetilde \pi}}
\newcommand{\tm}{{\tilde m}}
\newcommand{\ddelta}{\boldsymbol{\delta}}
\newcommand{\dl}{\underline{d}}
\newcommand{\du}{\overline{d}}

\renewcommand{\mod}{\text{ mod }}
\newcommand{\ac}{\operatorname{ac}}
\newcommand{\Br}{\cB_{|r}}
\newcommand{\Bl}{\cB_{|\ell}}
\newcommand{\jm}{m}

\newcommand{\Eig}{\mathrm{Eig}}
\newcommand{\Bap}{\mathrm{Bap}}

\newtheorem{theorem}{Theorem}[section]
\newtheorem{lemma}[theorem]{Lemma}
\newtheorem{proposition}[theorem]{Proposition}
\newtheorem{corollary}[theorem]{Corollary}

\theoremstyle{definition}
\newtheorem{definition}[theorem]{Definition}
\newtheorem{example}[theorem]{Example}

\theoremstyle{remark}
\newtheorem{remark}[theorem]{Remark}

\numberwithin{equation}{section}

\newcommand{\tsigma}{{\widetilde{\sigma}}}

\newcommand{\DD}{\overline D_1}
\newcommand{\dd}{\overline d_1}

\begin{document}

\title{Besicovitch covering numbers for $\cB$-free and other shifts}

\author[1]{Stanis\l{}aw Kasjan}
\author[2]{Gerhard Keller\thanks{The authors thank Nicolae Strungaru for an enlightening E-mail exchange pointing out the relevance of his joint work (published and in preparation) with D.~Lenz and T.~Spindeler for the general framework of this paper as described in Section~\ref{sec:introduction}.  The authors also thank Christoph Richard for pointing out the relevance of reference \cite{Wiener-Wintner-2} by Wiener and Wintner.}}
\affil[1]{\small Faculty of Mathematics and Computer Science, Nicolaus Copernicus University, Chopina 12/18, 87–100 Toruń, Poland, Email: skasjan@mat.umk.pl}
\affil[2]{\small Department of Mathematics, University of Erlangen-N\"urnberg,
 Cauerstr. 11, 91058 Erlangen, Germany, Email: keller@math.fau.de}
\date{Version of \today}


\maketitle
\begin{abstract}
For a finite alphabet $A$ define by
$d_1(x,y):=\limsup_{n\to\infty}\frac{1}{2n+1}\#\{|i|\le n: x_i\neq y_i\}$
the \emph{Besicovitch pseudo-metric} on $A^\Z$. It is well known that a closed subshift of $A^\Z$ has finite covering numbers w.r.t.~$d_1$ if and only if it is mean-equicontinuous. Here we study, more generally, the scaling behaviour of these covering numbers for individual orbits which are generic for an ergodic measure $\mu$ on $A^\Z$ with discrete spectrum, and we explore their usefulness as invariants for block code equivalence. We illustrate this by developing tools to determine these covering numbers for various classes of $\cB$-free numbers (in particular also for square-free numbers), and we provide a continuous family of measures $\mu_s$, all with the same discrete spectrum generated by a single number, but such that $\mu_s$- and $\mu_{s'}$-typical $x$ resp. $x'\in A^\Z$ have so different covering numbers that there are no finite block codes mapping $x\to x'$ and $x'\to x$. (Indeed, both orbits have different amorphic complexities.)
\end{abstract}

\blfootnote{\emph{MSC 2020 clasification:} 37A35, 37A44}

\tableofcontents

\section{Introduction}\label{sec:introduction}

\subsection{The Besicovitch metric on subshifts}\label{subsec:intro-Besicovitch}

Let $A$ be a finite alphabet, $\# A\ge2$, and denote by $A^\Z$ the full two-sided shift over $A$. Define the \emph{Besicovitch pseudo-metric}
\begin{equation*}
d_1(x,y):=\limsup_{n\to\infty}\frac{1}{2n+1}\#\{|i|\le n: x_i\neq y_i\}
\end{equation*}
on $A^\Z$.
Factoring by the equivalence relation $x\sim y$ $:\Leftrightarrow$ $d_1(x,y)=0$ yields a factor space $A^\Z_\sim$ with factor map $\tpi:x\mapsto \tx$ and metric $\td_1(\tx,\ty):=d_1(x,y)$. The pseudo-metric $d_1$ is motivated by classical work on almost periodic functions, see \cite{Blanchard1997} for some more background and for the following results:
\begin{quote}
The space $(A^\Z_\sim,d_1)$ is a complete metric space. It is neither separable nor locally compact, but it is pathwise connected and infinite-dimensional. The shift map $\sigma$ on $A^\Z$ determines a shift
$\tsigma$ on $A^\Z_\sim$ by $\tsigma\tx=\widetilde{\sigma x}$, and $\tsigma:A^\Z_\sim\to A^\Z_\sim$ is an isometry w.r.t.~the metric $\td_1$ on $A^\Z_\sim$.
\end{quote}
Moreover, as $d_1:A^\Z\times A^\Z\to[0,\infty)$ is Borel measurable, the factor map $\tpi:A^\Z\to A^\Z_\sim$ is measurable w.r.t. the $\sigma$-algebra generated by open $\td_1$-balls in $A^\Z_\sim$, which we denote by $\tAb$.
 As $A^\Z_\sim$ is not separable, one cannot expect that $\tAb$ is the full Borel-$\sigma$-algebra $\tAB$ on $(A_\sim^\Z,d_1)$. On
separable
subspaces of $A^\Z_\sim$, in particular on orbit closures, the traces of $\tAb$ and of $\tAB$ coincide however.
This suggests to restrict to invariant subshifts $X$ (not necessarily closed) of $A^\Z$ and to study
$$
X_\sim:=\tpi(X)
$$
or its closure w.r.t.~$\td_1$ and to investigate properties of the dynamical system $(A^\Z_\sim,\tsigma,\\\mu\circ\tpi^{-1})$ derived from an ergodic system $(A^\Z,\sigma,\mu)$. This  complements classical results by Wiener and Wintner \cite[Thms.~6 - 8]{Wiener-Wintner-2} on Besicovitch almost periodicity,
see our Theorems~\ref{theo:basic}(i) and~\ref{theo:precisions}. In particular, Theorem~\ref{theo:precisions}b helps to check whether a specific orbit is Besicovitch almost periodic. Generalizations of \cite{Wiener-Wintner-2}\footnote{The reader may also consult Veech's commentary on \cite{Wiener-Wintner-2} in the Collected Works of Norbert Wiener.} and more details are provided in the recent paper \cite{Lenz2023}. Because of this detailed exposition we chose to cite from \cite{Lenz2023} when we need some precise facts for our proofs, being fully aware that in the case of $\Z$-actions much of what is needed is contained in \cite{Wiener-Wintner-2}.

\emph{Notation for closures:} The closure of $X$ in $A^\Z$ w.r.t.~the usual product topology is denoted by $\overline{X}$ and that of $X_\sim$ in $A^\Z_\sim$ w.r.t.~the topology determined by the metric $\td_1$ is denoted by $\overline{X_\sim}$. So we use the same closure-bar in both topologies, but it will always be clear from the context to which topology it refers.

\subsection{Besicovitch covering numbers}\label{subsec:Bcn}

Some aspects of the metric geometry of a space $(X,d_1)$, equivalently of $(X_\sim,\td_1)$, are captured by (the scaling behaviour of) its covering numbers:
\begin{definition}\label{def:covering-numbers}
The covering number $\cN_\epsilon(X)$ of $X$ is the smallest number of open $d_1$-balls of radius $\epsilon$ centered at some point of $X$ which is necessary to cover $X$.
It coincides with $\cN_\epsilon(X_\sim)$, the smallest number of open $\td_1$-balls of radius $\epsilon$ centered at some point of $X_\sim$ which is necessary to cover $X_\sim$.

Observe that, for each real number $M>1$,
\begin{equation*}
\cN_{M\epsilon}(\overline{X_\sim})\le\cN_\epsilon(X_\sim)=\cN_\epsilon(X)\le \cN_{M^{-1}\epsilon}(\overline{X_\sim}).
\end{equation*}
\end{definition}
Recall that $(X_\sim,\td_1)$ is relatively compact if and only if $\cN_\epsilon(X)<\infty$ for all $\epsilon>0$. Observe also that $(X_\sim,\td_1)$ is relatively compact if and only if $(\overline{X_\sim},\td_1)$ is compact, because the ambient space $A^\Z_\sim$ is complete.

In this paper we attempt to study covering numbers $\cN_\epsilon(\cO(x))$ of shift orbits $\cO(x):=\{\sigma^nx:n\in\Z\}$ - with special emphasis on the case where $x$ is the indicator function of a set of $\cB$-free numbers, see Section~\ref{sec:B-free}. Here are two different approaches:
\begin{enumerate}[a)]
\item
Fix any $x\in A^\Z$ and consider the closed subshift $X=\overline{\cO(x)}$, where the closure is taken in the usual product topology on $A^\Z$. This is the situation for which the numbers $\cN_\epsilon(X)$ and their scaling behaviour were studied in
\cite{FGJ-2016, FGJK-2021}.\,\footnote{The approach in both these papers is based on separation numbers, indeed. But, as noted in \cite[Sec.~3.2]{FGJ-2016}, this is equivalent to using spanning numbers which, in turn, are identical to our covering numbers.}

Recall that $(X,\sigma)$ is \emph{mean equicontinuous}, if
the factor map $\tpi:X\to(X_\sim,\td_1)$ is continuous. In this case
$(X_\sim,\td_1)$ is a compact metric space.

For a minimal topological dynamical system $(X,T)$ the following are equivalent:
\begin{compactenum}[(i)]
\item $(X,T)$ is mean equicontinuous.
\item The $d_1$-covering numbers $\cN_\epsilon(X)$, $\epsilon>0$, are finite.
\item $(X,T)$ is uniquely ergodic (with a measure $\mu$), and the factor map from $(X,T)$ to its maximal equicontinuous factor (MEF) is a measurable isomorphism between $(X,T,\mu)$ and the MEF equipped with its unique invariant measure. (In short: $(X,T)$ is an \emph{isomorphic extension} of its MEF.)
\end{compactenum}
For (i) $\Leftrightarrow$ (ii) see \cite[Cor.~3.7]{FGJK-2021}. The equivalence (i) $\Leftrightarrow$ (iii) follows from \cite[Thm.~2.1]{DG15} and \cite[Thm.~3.8]{Li2015}, as was noticed  in \cite[Thm.~3.2]{Garcia-Ramos2019}

If $X=\overline{\cO(x)}$ is minimal and mean equicontinuous, the set $X_\sim$
carries the structure of a compact abelian group isomorphic to the MEF \cite[Thm.~3.5]{Li2015}, see also \cite[Prop.~49]{Garcia-Ramos2017} and \cite[Lem.~3.5]{Garcia-Ramos2019}.

\item Fix again $x\in A^\Z$ and consider $X=\cO(x):=\{\sigma^nx: n\in\Z\}$. Then $X_\sim=\cO(\tx)$, and this set is relatively compact w.r.t.~$\td_1$ if and only if $x$ is mean almost periodic\,\footnote{Nicolae Strungaru (private communication): For the group $G=\Z$ this can be deduced from \cite{Moody2004}, for more general groups $G$ this is work in progress by Lenz, Spindeler and Strungaru.
A short direct argument is provided in Lemma~\ref{lemma:compact-mean-a-p} below.}, i.e.~if for each $\epsilon>0$ the set $\{i\in\Z: d_1(x,\sigma^ix)<\epsilon\}$ is relatively dense in $\Z$ (which means in this case that the set has bounded gaps), see e.g.~Definition 3.1 in \cite{Lenz2023}.

We stress that the mean almost periodic $x\in A^\Z$ are precisely those elements $x\in A^\Z$, for which all numbers $\cN_\epsilon(\cO(x))$ are finite, such that their scaling behaviour in the limit $\epsilon\to 0$ can be studied.

If $x$ is generic for some ergodic invariant probability measure $\mu$ on $A^\Z$, then  $x$ is mean almost periodic if and only if the dynamical system $(A^\Z,\sigma,\mu)$ has discrete spectrum (in which case all generic points for $\mu$ are mean almost periodic), see \cite[Thm.~3.8]{Lenz2023}.

In this situation denote by $H$ the compact abelian group which is dual to the discrete group of eigenvalues of $(A^\Z,\sigma,\mu)$. In view of the Halmos-von Neumann theorem the system $(A^\Z,\sigma,\mu)$ is measurably isomorphic to an ergodic rotation on $H$ (equipped with Haar measure), see e.g. \cite[Thm.~3.6]{walters}. Moreover, since $\tsigma$ acts isometrically on the $\td_1$-orbit closure $\overline{\cO(\tx)}$, this set  carries the structure of a compact abelian group,
and this group is isomorphic to $H$.\,\footnote{Nicolae Strungaru (private communication): The isomorphy can be deduced from work in progress by Lenz, Spindeler and Strungaru. See also Theorem~\ref{theo:basic} below.} Since genericity of an orbit is preserved under $d_1$-convergence, all orbits $y\in A^\Z$ with $\tilde{y}\in\overline{\cO(\tx)}$ are generic for the measure $\mu$. But note that if $(\sigma,\mu)$ is weakly mixing, then the set $\tpi^{-1}(\overline{\cO(\tx)})$ has $\mu$-measure zero, see Remark~\ref{remark:weak-mixing} below.
\end{enumerate}
We can summarize this discussion: Fix any $x\in A^\Z$.
\begin{equation}\label{eq:summary-product-closure}
\text{\parbox{\textwidth-1.6cm}{\textbf{For minimal $X=\overline{\cO(x)}$:}\\
The $d_1$-covering numbers $\cN_\epsilon(\overline{\cO(x)})$ are finite for all $\epsilon>0$
if and only if $(X,\sigma)$ is mean equicontinuous
 if and only if $(X,\sigma)$ is an isomorphic extension of its MEF. (The closure of $\cO(x)$ is taken in $X$ w.r.t. the product topology.) In this case,
$x$ is generic for the unique invariant measure of $(X,\sigma)$, and
$(\overline{\cO(x)})_\sim$ carries the structure of a compact abelian group isomorphic to the MEF.}}
\end{equation}
\begin{equation} \label{eq:summary-Besicovitch-closure}
\text{\parbox{\textwidth-1.6cm}{\textbf{For $X=\cO(x)$ where $x$ is generic for an ergodic invariant probability $\mu$:}\\
The $\td_1$-covering numbers $\cN_\epsilon(\cO(\tx))$ and $\cN_\epsilon(\overline{\cO(\tx)})$ are finite for all $\epsilon>0$ if and only if the dynamical system $(A^\Z,\sigma,\mu)$ has discrete spectrum. (The closure of $\cO(\tx)$ is taken in $(X_\sim,\td_1)$.) In this case, $\overline{\cO(\tx)}$ carries the structure of a compact abelian group, and the system $(A^\Z,\sigma,\mu)$ is measurably isomorphic  to an ergodic rotation on $\overline{\cO(\tx)}$.}}
\end{equation}

\begin{remark}\label{remark:minimal-case}
Consider some $x\in A^\Z$ for which
$(X,\sigma)=(\overline{\cO(x)},\sigma)$ is a minimal dynamical system, where the closure is taken in the product topology.
Suppose now that $X$ has finite covering numbers $\cN_\epsilon(X)$. In view of item a) above it follows that $(X,\sigma)$ is mean equicontinuous, in particular $(X_\sim,\td_1)$ is a compact metric space on which the shift $\tsigma$ acts minimally. It follows that $\cO(\tx)$ is dense in $X_\sim$, so that
$\overline{\cO(\tx)}=X_\sim=(\overline{\cO(x)})_\sim$. In particular,
\begin{equation*}
\cN_\epsilon(\cO(x))=\cN_\epsilon(\cO(\tx))=
\cN_\epsilon(\overline{\cO(\tx)})
=\cN_\epsilon((\overline{\cO(x)})_\sim)=
\cN_\epsilon(\overline{\cO(x)})
\end{equation*}
for all $\epsilon>0$. Hence, for such $x$, the approach sketched in b) covers the situation sketched in a).
\end{remark}

\subsection{Besicovitch covering numbers as block code invariants}\label{subsec:sliding-block-code}

Suppose that $y\in A^\Z$ can be obtained from $x\in A^\Z$ by applying a block code $\cC$ of length $M$, i.e. $y=F_\cC(x)$ for the factor map $F_\cC:\overline{\cO(x)}\to\overline{\cO(y)}$ determined by $\cC$.\,\footnote{\label{foot:BC} We recall: A block code $\cC$ of length $M=2m+1$ is a map $\cC:A^M\to A$. It defines a continuous map $F_\cC:A^\Z\to A^\Z$, $(F_\cC(v))_i=\cC(v_{i-m}\dots v_{i+m})$. Conversely, for each continuous map $F:\overline{\cO(x)}\to\overline{\cO(y)}$ which commutes with the shift there is a block code $\cC$ such that $F$ is the restriction of $F_\cC$ to $\overline{\cO(x))}$ (Hedlund's theorem \cite{Hedlund1969}).}
It is easy to see that $d_1(F_\cC u,F_\cC v)\le M\cdot d_1(u,v)$ for all $u,v\in\overline{\cO(x)}$,
so that one can define
\begin{equation*}
\tilde{F}_\cC:(\overline{\cO(x)})_\sim\to (\overline{\cO(x)})_\sim,\; \tilde{u}\mapsto\widetilde{F_\cC u},
\end{equation*}
which is $M$-Lipschitz. In particular, $\tilde{F}_\cC$ can be uniquely extended to a $M$-Lipschitz map on $\overline{(\overline{\cO(x)})_\sim}$ and then restricted to  a $M$-Lipschitz map from $\overline{\cO(\tx)}$ onto $\overline{\cO(\tilde{y})}$.
It follows that $\cN_{M\epsilon}(\overline{\cO(\tilde{y})})\le \cN_{\epsilon}(\overline{\cO(\tx)})$ and hence $\cN_{M\epsilon}(\cO(\tilde{y}))\le \cN_{\epsilon}(\cO(\tx))$.
If $x$ and $y$ are \emph{block code equivalent}, i.e.~if they can be mutually obtained from each other by sliding block codes of length at most $M$ \footnote{Note that $x$ and $y$ are block code equivalent if and only if their orbit closures are isomorphic subshifts and there is an isomorphism sending $x$ to $y$.}, then
\begin{equation}\label{eq:covering-number-comparison}
\cN_{M\epsilon}(\cO(\tilde{y}))\le \cN_{\epsilon}(\cO(\tx))\le \cN_{M^{-1}\epsilon}(\cO(\tilde{y}))\quad\text{for all $\epsilon>0$.}
\end{equation}
This motivates to define the following equivalence relation on the set of all monotonically decreasing functions from $(0,\infty)$ to $\N$:
\begin{equation*}
F\approx G\quad\text{if and only if}\quad\exists M>0\ \forall \epsilon>0:\ F(M\epsilon)\le G(\epsilon)\le F(M^{-1}\epsilon).
\end{equation*}
Now we can reformulate observation~\eqref{eq:covering-number-comparison} by saying:
\begin{equation}\label{eq:approx-equivalence}
\text{\parbox{\textwidth-1.6cm}{The $\approx$-equivalence class
$[\cN_\bullet(\cO(x))]_\approx=[\cN_\bullet(\cO(\tx))]_\approx$ of a sequence $x\in A^\Z$ is an invariant for block code equivalence.
}}
\end{equation}
Following \cite{FGJ-2016, FGJK-2021}, we call the
polynomial growth rate of these functions of $\epsilon$ the \emph{amorphic complexity} of $\cO(x)$, more precisely
\begin{equation}\label{eq:ac}
\ac(\cO(x))=\alpha\in [0,\infty],\;\text{ if }\;
\cN_\epsilon(\cO(x))\cdot\epsilon^t
\to
\begin{cases}
0&\text{ for }t>\alpha\\
+\infty&\text{ for }t<\alpha
\end{cases}
\quad\text{ as }\epsilon\to0.
\end{equation}
Quite obviously, $\ac(\cO(x))$ is an invariant for block code equivalence. It is nothing but the Minkowski dimension of $\cO(\tx)$ w.r.t. the Besicovitch metric.
Therefore, just as in \cite{Kloeckner2012,Helfter2022}, one can study other scaling regimes of $\cN_\epsilon(\cO(x))$ and determine the corresponding critical exponents, see Subsection~\ref{subsec:Erdos-examples} for an example.

{Beyond the limited setting of $\Z$-actions studied in the present paper, Baake et al.~\cite{BGG2025} determined precisely the amorphic complexity (which they call \emph{orbit separation dimension}) for quite a few spaces of self-similar aperiodic tilings of $\R^d$ that are generated by a primitive inflation rule, by reducing the problem to determining Lyapunov exponents of associated graph-directed iterated function systems.}

\begin{remark}[Codes whose length has finite moments of all orders]\quad\\
Critical exponents like the amorphic complexity or the exponent for power exponential growth used in Example~\ref{ex:square-free} (more generally, for all \emph{scalings} in the sense of \cite[Def.~2.1]{Helfter2022}) are not only invariants for (bounded length) block code equivalence, but they are also invariant under equivalence by block codes $\cC$ whose length has finite moments of all orders.
In this case the associated map $\tilde F_\cC$ is $\alpha$-Hölder for each exponent $\alpha<1$, and this suffices for the claimed invariance.
\end{remark}

\subsection{Plan of the paper}

The aim of this work is to extend the applicability of the Besicovitch (pseudo-)metric  beyond the class of mean equicontinuous systems to individual mean almost periodic orbits, to  develop methods of approximating the amorphic complexity of such orbits and to exhibit applications of the amorphic complexity as a block code invariant that may be stronger than other known invariants  in some classes of examples.

In Section~\ref{sec:Besicovitch space} we provide some general facts on dynamics on  spaces $X_\sim=\tpi(X)$ where $X$ is a shift invariant subset of $A^\Z$. The emphasis lies on characterizing cases when $\overline{X_\sim}$ is compact, see Theorem~\ref{theo:basic} for a combination of results from the literature with some new aspects. Given an ergodic shift-invariant measure $\mu$, Theorem~\ref{theo:precisions} contains in particular a new description of the set of points which are $\mu$-generic and Besicovitch almost periodic.\footnote{Vershik's paper \cite{Vershik2012} contains a number of concepts and results which are close to the ones from Section~\ref{sec:Besicovitch space} - but seem different nontheless. A closer look at these connections should be fruitful to study the case $\overline{D}_1(\mu)>0$, see Proposition~\ref{proposition:R-D1}.}

As applications, we focus on two classes of examples where $x\in\{0,1\}^\Z$ is the indicator function of a weak model set with Borel window: $\cB$-free numbers in Section~\ref{sec:B-free} and sequences $x\in\{0,1\}^\Z$ generated by a golden mean rotation in Section~\ref{sec:rotation}.

Let $x=\eta_\cB$ be the indicator function of a set of $\cB$-free numbers. If we follow approach \eqref{eq:summary-product-closure}, we must restrict to $\eta_\cB$ which are of Toeplitz type, and even more to the uniquely ergodic ones in this class\,\footnote{See \cite[Cor.~1.4]{BKKL2015} and \cite[Thm.~B]{Dymek2023} for the fact that minimality of the $\cB$-free subshift and the Toeplitz property of the sequence $\eta_\cB$ are equivalent, and \cite{Keller2022} for the observation that not all these systems are uniquely ergodic, although there are uniquely ergodic ones among the irregular Toeplitz examples.}.
In Subsection~\ref{subsec:Toeplitz-examples} we present a parametrized family of such shifts $X_{\eta_\cB}:=\overline{\cO(\eta_\cB)}$ with mutually distinct amorphic complexities. Hence these shifts are pairwise not block code equivalent (in particular non-isomorphic), although they all have isomorphic MEFs, trivial automorphisms groups and (of course) entropy zero.

Approach \eqref{eq:summary-Besicovitch-closure}, in contrast, allows to treat all \emph{Besicovitch sets} $\cB$, i.e.~all sets $\cB\subseteq\N$ such that $\eta_\cB$ is generic for an invariant probability measure $\mu$ - the so-called Mirsky measure of $X_{\eta_\cB}$ that has always discrete spectrum \cite[Thm.~2(a)]{KR2015}.
If $\cB$ is taut\,\footnote{Tautness is a mild regularity property of $\cB$, see Section~\ref{sec:B-free}.},
$(X_{\eta_\cB},\sigma,\mu)$  is actually measurably isomorphic to the rotation by $1$ on the odometer group $G_\cB$ determined by $\cB$ \cite[Thm.~F]{BKKL2015}.
For the classical square-free shift we show that $\eta_\cB$ has infinite amorphic complexity, but critical value $\alpha=1$ on the \emph{power exponential scale} $\epsilon\mapsto\exp(-\epsilon^{-\alpha})$.

In Section~\ref{sec:rotation} we present a parametrized family of weak model sets with (non-compact) Borel windows (based on the golden rotation) such that, for each fixed parameter $s>1$, Lebesgue-a.e. $h\in\T^1$ generates a sequence $x\in\{0,1\}^\Z$ that is generic for an invariant ergodic measure $\mu_s$ with full topological support and for which the system $(\{0,1\}^\Z,\sigma,\mu)$ is isomorphic to the underlying rotation (which is the same for all parameters $s$), but has amorphic complexity $\frac{s}{s-1}$. Hence, for two different parameters, almost all $h\in\T^1$ generate non block code equivalent sequences $x$ with identical orbit closures and identical spectra.

Some proofs are collected in Section~\ref{sec:proofs}.

\section{Invariant measures on Besicovitch space}\label{sec:Besicovitch space}

\subsection{A decomposition of $A^\Z_\sim$ and almost periodicity}

For $x,y\in A^\Z$ let $\tx=\tpi x$ and $\ty=\tpi y$. Define
\begin{equation*}
\begin{split}
D_1(x,y)
:=&
\inf\{\td_1(\tilde u,\tilde v): \tilde u\in\overline{\cO(\tx)}, \tilde v\in\overline{\cO(\tilde y)}\}\\
=&
\inf\{\td_1(\tsigma^m\tx,\tsigma^n\tilde y): m,n\in\Z\}\\
=&
\inf\{\td_1(\tx,\tsigma^n\tilde y): n\in\Z\}
=
\inf\{d_1(x,\sigma^n y): n\in\Z\}
\end{split}
\end{equation*}
where the first identity follows because the closures are taken w.r.t.~$\td_1$, and the second one because $\tsigma$ leaves $\td_1$ invariant. Observe that $D_1$ is Borel measurable and $\sigma$-invariant in both variables separately. Moreover, $D_1$ is a pseudo-metric\,\footnote{Indeed: Let $x,y,z\in A^\Z$ and $\delta>0$. There are $i,j\in\Z$ such that $d_1(x,\sigma^i y)<D_1(x,y)+\delta$ and $d_1(y,\sigma^j z)<D_1(y,z)+\delta$. Hence $D_1(x,z)\le d_1(x,\sigma^{i+j} z)\le d_1(x,\sigma^i y)+d_1(\sigma^i y,\sigma^i(\sigma^j z))<D_1(x,y)+D_1(y,z)+2\delta$.}.

\begin{lemma}\label{lemma:dichotomy}
Let $x,y\in A^\Z$. Then
\begin{compactenum}[a)]
\item $D_1(x,y)=0$\; $\Leftrightarrow$\; $\overline{\cO(\tx)}=\overline{\cO(\tilde y)}$.
\item $D_1(x,y)>0$\; $\Leftrightarrow$\; $\overline{\cO(\tx)}\cap\overline{\cO(\tilde y)}=\emptyset$.
\item $\{\overline{\cO(\tx)}:x\in A^\Z\}$ is a partition of $A^\Z_\sim$ into closed, $\tAb$-measurable, $\tsigma$-invariant sets.
\end{compactenum}
\end{lemma}

\begin{proof}
It suffices to prove the $\Rightarrow$-implications in a) and b) and the measurability in c).\\
a)\;Suppose that $D_1(x,y)=0$. Because of the symmetry in $x$ and $y$ and the $\tsigma$-invariance of $\overline{\cO(\ty)}$ it suffices to prove that $\tx\in\overline{\cO(\ty)}$, and the latter follows from $D_1(x,y)=0$.\\
b)\; If $\overline{\cO(\tx)}\cap\overline{\cO(\tilde y)}\neq\emptyset$, then $D_1(x,y)=0$ by definition.\\
c)\; Fix $x\in A^\Z$. Then
\begin{equation*}
\overline{\cO(\tx)}
=
\{\ty: y\in A^\Z, D_1(x,y)=0\}
=
\bigcap_{k\in\N}\bigcup_{n\in\Z}\{\ty:y\in A^\Z, \td_1(\ty,\tsigma^n\tx)<1/k\}.
\end{equation*}
\end{proof}

\begin{lemma}\label{lemma:generic}
Let $x,y\in A^\Z$ and $f\in C(A^\Z,\C)$.
\begin{compactenum}[a)]
\item If $f$ depends only on coordinates in $[-N,N)$, then
\begin{equation*}
\limsup_{n\to\infty}\frac{1}{2n+1}\sum_{|k|\le n}|f(\sigma^ky)-f(\sigma^{k}x)|
\le
4N\cdot\|f\|_\infty\cdot d_1(y,x).
\end{equation*}
\item If $D_1(x,y)=0$, then
\begin{equation}\label{eq:generic-1}
\inf_{\ell\in\Z}\limsup_{n\to\infty}\frac{1}{2n+1}\sum_{|k|\le n}|f(\sigma^ky)-f(\sigma^{k+\ell}x)|=0
\end{equation}
and
\begin{equation}\label{eq:generic-2}
\lim_{n\to\infty}\frac{1}{2n+1}\left|\sum_{|k|\le n}(f(\sigma^ky)-f(\sigma^kx))\right|=0 .
\end{equation}
\end{compactenum}
\end{lemma}

\begin{proof}
a)\; This is an elementary estimate.\\
b)\;Assertion \eqref{eq:generic-1} for functions $f$ that depend on finitely many coordinates only follows from a), because $\inf_{\ell\in\Z}d_1(\sigma^\ell x,y)=D_1(x,y)=0$.
It extends to general $f\in C(A^\Z,\C)$, since these functions are $\|\,.\,\|_\infty$-dense in $C(A^\Z,\C)$. This implies assertion \eqref{eq:generic-2}, because
\begin{equation*}
\limsup_{n\to\infty}\frac{1}{2n+1}\left|\sum_{|k|\le n}(f(\sigma^ky)-f(\sigma^kx))\right|
=
\lim_{n\to\infty}\frac{1}{2n+1}\left|\sum_{|k|\le n}(f(\sigma^ky)-f(\sigma^{k+\ell}x))\right|
\end{equation*}
for every $\ell\in\Z$.
\end{proof}

\begin{definition}\label{defi:almost periodic}
\begin{enumerate}[a)]
\item
A point $x\in A^\Z$ is \emph{mean almost periodic}, if, for every $\epsilon>0$, the set $\{k\in\Z: d_1(x,\sigma^kx)<\epsilon\}$ has bounded gaps, see e.g.~\cite[Def.~3.1]{Lenz2023}. This is the same as saying that $\tpi x$ is almost periodic under $\tsigma$.
(Observe that in \cite{Lenz2023} our pseudo-metric $d_1$ is denoted by $D$.)
\item
A point $x\in A^\Z$ is \emph{Besicovitch almost periodic} if for each $f\in C(A^\Z,\C)$ and each $\epsilon>0$ there is a finite linear
combination~$\Xi_\epsilon=\sum_{j\in J}c_j\xi_j$
of characters $\xi_j:\Z\to\C$ such that
(see e.g.~\cite[Def.~4.1]{Lenz2023})
$$
\limsup_{n\to\infty}\frac{1}{2n+1}\sum_{|k|\le n}|f(\sigma^kx)-\Xi_\epsilon(k)|<\epsilon.
$$
\end{enumerate}
\end{definition}

\begin{lemma}\label{lemma:D_1-Besicovitch}
Let $x,y\in A^\Z$ and suppose that $D_1(x,y)=0$. Then $x$ is mean (resp.~Besicovitch) almost periodic if and only $y$ is mean (resp.~Besicovitch) almost periodic.
\end{lemma}

\begin{proof}
The ``mean assertion'' is nearly trivial. We turn to the ``Besicovitch assertion'':
Suppose $x$ is Besicovitch almost periodic, $f\in C(A^\Z,\C)$, and $\Xi_\epsilon=\sum_{j\in J}c_j\xi_j$ as in Definition~\ref{defi:almost periodic}b). As $D_1(x,y)=0$, there are a sequence $(\ell_i)_i$ of integers and some $a_j\in\C$ $(j\in J)$ such that
$$
\Delta_i:=\limsup_{n\to\infty}\frac{1}{2n+1}\sum_{|k|\le n}|f(\sigma^ky)-f(\sigma^{k+\ell_i}x)|\to0\quad\text{as }i\to\infty
$$
by Lemma~\ref{lemma:generic}b)
and $\lim_{i\to\infty}\xi_j(\ell_i)= a_j$ $(j\in J)$. Let $\Xi=\sum_{j\in J}c_ja_j\xi_j$. Then
\begin{equation*}
\begin{split}
&\limsup_{n\to\infty}\frac{1}{2n+1}\sum_{|k|\le n}|f(\sigma^ky)-\Xi(k)|\\
&\le
\Delta_i+
\limsup_{n\to\infty}\frac{1}{2n+1}\sum_{|k|\le n}|f(\sigma^{k+\ell_i}x)-\Xi_\epsilon(k+\ell_i)|\\
&\hspace*{12mm}+
\limsup_{n\to\infty}\frac{1}{2n+1}\sum_{|k|\le n}|\Xi_\epsilon(k+\ell_i)-\Xi(k)|\\
&\le
\Delta_i+\epsilon+\limsup_{n\to\infty}\frac{1}{2n+1}\sum_{|k|\le n}\sum_{j\in J}|c_j(\xi_j(\ell_i)-a_j)\xi_j(k)|\\
&\to\epsilon\quad\text{as }i\to\infty\ ,
\end{split}
\end{equation*}
so that also $y$ is Besicovitch almost periodic.
\end{proof}

\subsection{Another look at Besicovitch almost periodicity}

The next three lemmas are kind of folklore knowledge, but since we need some precise details and since $(A^\Z_\sim,\td_1)$ is not even locally compact, we provide proofs. Observe also that a subset of a complete metric space is relatively compact if and only if it is totally bounded.

\begin{lemma}\label{lemma:relatively-compact}
Suppose $(M,d)$ is a metric space, $T:M\to M$ an isometry, and $x\in M$. If $\overline{\cO(x)}$ admits a $T$-invariant Borel probability measure $\mu$, then $\cO(x)$ is totally bounded, and
the maximal cardinality of a $2\epsilon$-separated set $R_\epsilon\subseteq\cO(x)$ is bounded by $\mu(B_d(x,\epsilon))^{-1}<\infty$.
\end{lemma}
\begin{proof}
For all $\epsilon>0$ and $n\in\Z$ we have
$\mu(B_d(T^nx,\epsilon))=\mu(B_d(x,\epsilon))$. Suppose for a contradiction that there exists $\epsilon>0$ such that $\mu(B_d(x,\epsilon))=0$. Then
$1=\mu(\overline{\cO(x)})\le \mu\left(\bigcup_{n\ge 0}B_d(T^nx,\epsilon) \right)
=0$. Hence
$\mu(B_d(x,\epsilon))>0$ for all $\epsilon>0$.

Suppose now that $R_\epsilon\subseteq \N$ is such that the points $T^nx\ (n\in R_\epsilon)$ are $2\epsilon$-separated. Then
\begin{equation*}
1\ge \mu\left(\bigcup_{n\in R_\epsilon}B_d(T^nx,\epsilon)\right)
=
\sum_{n\in R_\epsilon}\mu(B_d(T^nx,\epsilon))
=
\card R_\epsilon\cdot\mu(B_d(x,\epsilon)),
\end{equation*}
so that $\card R_\epsilon\le{(\mu(B_d(x,\epsilon)))^{-1}}<\infty$.
Consider such a set $R_\epsilon$ with maximal cardinality.
Then $\cO(x)\subseteq\bigcup_{n\in R_\epsilon}B_d(T^nx,2\epsilon)$.
\end{proof}

The following fact is well known, see e.g.~the discussion in \cite[Sec.~5.1]{Vershik2012}.
\begin{lemma}\label{lemma:compact-mean-a-p}
A point $x\in A^\Z$ is mean almost periodic if and only if $\cO(\tpi x)$ is relatively compact (w.r.t.~$\td_1$).
In this case, $(\overline{\cO(\tpi x)},\td_1,\tsigma)$ is an invertible, compact, isometric dynamical system with a unique invariant probability measure, call it~$\tm$.
\end{lemma}

\begin{lemma}\label{lemma:compact-ONB}
Suppose that the dynamical system $(A^\Z,\sigma,\mu)$ is ergodic and that
$\cO(\tpi x)$ is relatively compact for some $\mu$-generic point $x\in A^\Z$. Then
\begin{compactenum}[a)]
\item $(A^\Z,\sigma,\mu)$ has discrete spectrum, and
\item all eigenvalues of $(A^\Z,\sigma,\mu)$ are eigenvalues of $(\overline{\cO(\tpi x)},\tsigma,\tm)$.
\end{compactenum}
\end{lemma}
\begin{proof}
We are going to show the following technical statement:
\begin{equation}\label{eq:tech-statement}
\text{\parbox{\textwidth-2cm}
{For every $f\in L^2_\mu$ and
every $\epsilon>0$ there are eigenvalues $\xi_1,\dots,\xi_n$ of
$(\overline{\cO(\tpi x)},\tsigma,\tm)$ with the following properties: $\xi_1,\dots,\xi_n$ are also eigenvalues of
$(A^\Z,\sigma,\mu)$ with normalized eigenfunctions $e_{\xi_1},\dots,e_{\xi_n}$, and there are $c_1,\dots,c_n\in\C$ such that
$\|f-\textstyle\sum_{i=1}^nc_ie_{\xi_i}\|_2<\epsilon$.
}}
\end{equation}
Hence the eigenfunctions of $(A^\Z,\sigma,\mu)$ span $L^2_\mu$ (this is a), and  each eigenfunction $f$ of $(A^\Z,\sigma,\mu)$ belongs to the span of those eigenfunctions of $(A^\Z,\sigma,\mu)$ with eigenvalues that are also eigenvalues of $(\overline{\cO(\tpi x)},\tsigma,\tm)$ (this yields b).

It suffices to verify \eqref{eq:tech-statement} for a set of continuous functions $f$ whose linear span is dense in $L^2_\mu$, here for the family of functions $1_K$ where the sets $K\subseteq A^\Z$ are cylinder sets.
We borrow some arguments from \cite[Sec.~5.3]{Moreira-notes}:
Denote $\tX:=\overline{\cO(\tx)}$ where $\tx=\tpi(x)$.
Define $\Gamma_K:\cO(\tx)\to L^2_\mu$, $\Gamma_K(\tsigma^j\tx)=1_K\circ \sigma^j$. Then,
if $1_K$ depends only on coordinates $x_{-m},\dots,x_m$ and $M=2m+1$,
\begin{equation}\label{eq:embedding}
\begin{split}
\|\Gamma_K(\tsigma^i\tx)-\Gamma_K(\tsigma^j\tx)\|_2^2
&=
\mu(\sigma^{-i}K\triangle \sigma^{-j}K)\\
&=
\lim_{n\to\infty}\frac{1}{2n+1}\sum_{|k|\le n}|1_K(\sigma^{i+k}x)-1_K(\sigma^{j+k}x)|\\
&\le
M\cdot\td_1(\tsigma^i\tx,\tsigma^{j}\tx),
\end{split}
\end{equation}
where we used the $\mu$-genericity of $\tx$ for the second equality.
Hence $\Gamma_K$ extends uniquely to a continuous embedding $\Gamma_K:\tX\to L^2_\mu$, and $\|\Gamma_K(\tu)\|_2=\|1_K\|_2\le 1$ for all $\tu\in\tX$.

Recall the measure $\tm$ on $\tX$ from Lemma~\ref{lemma:compact-mean-a-p}. Define the linear map
\begin{equation*}
F_K:L^1_{\tilde m}\to L^2_\mu,\quad
F_K(\psi)=\int\overline{\psi(\tu)}\cdot\Gamma_K(\tu)\,d\tm(\tu),
\end{equation*}
and observe that $\|F_K\|\le 1$.
Pick $0\le\phi_n\in L^\infty_{\tilde m}$ such that $\supp(\phi_n)\subseteq B_{\td_1}(\tx,n^{-2})$ and $\int\phi_n\,d\tm=1$, and a linear combination $\psi_n$ of eigenfunctions for $\tsigma$ such that $\|\phi_n-\psi_n\|_{L^1_{\tilde m}}<n^{-1}$. Then, observing that $\Gamma_K(\tx)=1_K$,
\begin{equation*}
\begin{split}
&\|1_K-F_K(\psi_n)\|_2\\
&\le
\left\|\int{\phi_n(\tu)}\left(\Gamma_K(\tx)-\Gamma_K(\tu)\right)d\tm(\tu)\right\|_2
+
\|F_K(\phi_n)-F_K(\psi_n)\|_2
\\
&\le
\sup_{\td_1(\tu,\tx)\le n^{-2}}\|\Gamma_K(\tx)-\Gamma_K(\tu)\|_2
+
\|F_K\|\cdot\|\phi_n-\psi_n\|_{L^1_{\tilde m}}
\le
\sqrt M\cdot n^{-1}+n^{-1}.
\end{split}
\end{equation*}
If $\psi\circ\tsigma=z\,\psi$ in $L^1_{\tm}$, then $F_K(\psi)\circ \sigma=z\,F_K(\psi)$ in $L^2_\mu$, because $\Gamma_K(\tu)\circ \sigma=\Gamma_K(\tsigma\tu)$. Hence $F_K$ maps eigenfunctions to eigenfunctions for the same eigenvalue. In particular, each $F_K(\psi_n)$ is a linear combination of eigenfunctions of $(A^\Z,\sigma,\mu)$ with eigenvalues which are also eigenvalues of $(\overline{\cO(\tpi x)},\tsigma,\tm)$.
This is \eqref{eq:tech-statement} for $f=1_K$.
\end{proof}

\begin{proposition}\label{proposition:R-D1}
Suppose that the dynamical system $(A^\Z,\sigma,\mu)$ is ergodic.
\begin{compactenum}[a)]
\item There exists a constant $\DD(\mu)\in[0,1]$ such that $D_1(x,y)=\DD(\mu)$ for $\mu\times\mu$-a.a.~$(x,y)$.
\end{compactenum}
\quad\\[-3.5mm]
Suppose now that $\DD(\mu)=0$.
\begin{compactenum}[a)]
\setcounter{enumi}{1}
\item There exists a $D_1$-equivalence class $X_\mu\subseteq A^\Z$ such that $\mu(X_\mu)=1$ and all points in $X_\mu$ are generic for $\mu$.
\item $\cO(\tpi x)$ is relatively compact for all $x\in X_\mu$.
\item $\tX_\mu:=\tpi(X_\mu)=\overline{\cO(\tx)}$ for all $\tx\in \tX_\mu$, and $X_\mu=\tpi^{-1}(\tX_\mu)$.
\item $(\tX_\mu,\td_1,\tsigma)$ is a compact, minimal, isometric, uniquely ergodic dynamical system with invariant measure $\tilde\mu:=\mu\circ\tpi^{-1}$.
\item If $\tilde{e}_\xi$ is an eigenfunction of $(\tX_\mu,\tsigma,\tilde\mu)$ with eigenvalue $\xi$, then $\tilde{e}_\xi$ is continuous (meaning that it has a continuous version) and $e_\xi:=\tilde{e}_\xi\circ\tpi$ is an everywhere defined eigenfunction for $(X_\mu,\sigma,\mu)$ with $|e_\xi(x)|=1$ for all $x\in X_\mu$.
\end{compactenum}
\end{proposition}
\begin{proof}
a)\;
$D_1:X\times X\to[0,1]$ is measurable. Hence, for each $x\in X$, the function $D_{1,x}:y\mapsto D_1(x,y)$ is measurable. Since it is also $\sigma$-invariant and $\mu$ is ergodic, $D_{1,x}$ is $\mu$-a.s.~constant, say $D_{1,x}=f(x)$ $\mu$-a.s.~for each $x$. It follows that $f$ is $\mu$-a.s.~$\sigma$-invariant, say $f(x)=\DD(\mu)$ for $\mu$-a.e.~$x$, so that $D_1(x,y)=D_{1,x}(y)=f(x)=\DD(\mu)$ for $\mu\times\mu$-a.e.~$(x,y)$.\\[2mm]
b)\;
Suppose that $\DD(\mu)=0$. Denote by $X_\mu$ the set of all points $x$ such that $D_{1,x}(y)=0$ for $\mu$-a.e.~$y$. Then $\mu(X_\mu)=1$. If $x,x'\in X_\mu$, there is some $y\in X$ such that $D_1(x,y)=0=D_1(x',y)$, whence $D_1(x,x')=0$. Conversely, if $D_1(y,x)=0$ for some $x\in X_\mu$ and $y\in A^\Z$, then $D_{1,y}(x')=D_1(y,x')\le D_1(y,x)+D_1(x,x')=0$ for all $x'\in X_\mu$, and since $\mu(X_\mu)=1$, also $y\in X_\mu$. Finally, as $\mu(X_\mu)=1$, there exists at least one generic point $x\in X_\mu$, so that
all $y\in X_\mu$ are generic by Lemma~\ref{lemma:generic}b).
\\[2mm]
c)\;
Let $x\in X_\mu$.
Then, for any $x'\in X_\mu$ there are $n_1,n_2,\ldots\in\Z$ such that
$\td_1(\tpi x',\tsigma^{n_i}\tpi x)\to0$ as $i\to\infty$, so that $\tpi x'\in\overline{\cO(\tpi x)}$. Let $\tilde\mu=\mu\circ\tpi^{-1}$. Then $\tilde\mu$ is a probability measure on $(A^\Z_\sim,\tAb)$, invariant under~$\tsigma$, and
$\tilde{\mu} (\overline{\cO(\tpi x)})=\mu(\tpi^{-1}(\overline{\cO(\tpi x)}))\ge\mu(X_\mu)=1$. As $\overline{\cO(\tpi x})$ is separable, $\tilde{\mu}$ is actually a $\tsigma$-invariant Borel probability measure on $\overline{\cO(\tpi x})$, and the claimed relative compactness follows from Lemma~\ref{lemma:relatively-compact}.\\[2mm]
d)\;
Let $x,y\in X_\mu$, i.e.~$D_1(x,y)=0$, and denote $\tx=\tpi x$ and $\tilde y=\tpi y$. Then $\overline{\cO(\tx)}=\overline{\cO(\tilde y)}$ by Lemma~\ref{lemma:dichotomy}. In particular $\tpi(X_\mu)\subseteq\overline{\cO(\tx)}$ for all $x\in X_\mu$. Conversely, let $\tilde{y}=\tpi y\in\overline{\cO(\tx)}$ for some $y\in A^\Z$. Then $D_1(x,y)=0$, i.e. $y\in X_\mu$, so that $\tpi y\in\tpi(X_\mu)$. Hence $\tX_\mu=\overline{\cO(\tx)}$ for all $x\in X_\mu$. Finally, if $x\in\tpi^{-1}(\tX_\mu)$, then there is $x'\in X_\mu$ such that $\tpi(x')=\tpi(x)$, i.e.~$D_1(x,x')=0$, so that $x\in X_\mu$ because $X_\mu$ is a $D_1$-equivalence class.\\[2mm]
e)
\;Compactness follows from c), because $(A^\Z_\sim,\td_1)$ is complete, and minimality follows from d), so that the system is also uniquely ergodic since $\tsigma$ is an isometry.
The Borel measure $\tilde\mu$ was constructed in the proof of c).\\
f)
\;The continuity of $\tilde{e}_\xi$ follows from e), $e_\xi\circ \sigma=\tilde{e}_\xi\circ\tsigma\circ\tpi=\xi\,\tilde{e}_\xi\circ\tpi=\xi\,e_\xi$, and $|e_\xi(x)|=|\tilde{e}_\xi(\tpi x)|=1$ for all $x\in X_\mu$.
\end{proof}

The following theorem combines (special cases of) several of the main results from \cite{Lenz2023} with reasonings based on Proposition~\ref{proposition:R-D1}, but some parts are contained already in \cite{Wiener-Wintner-2} and \cite{Vershik2011}.

\begin{theorem}\label{theo:basic}
Suppose that the dynamical system $(A^\Z,\sigma,\mu)$ is ergodic. The following are equivalent:
\begin{compactenum}[(i)]
\item $\DD(\mu)=0$.
\item $(A^\Z,\sigma,\mu)$ has discrete spectrum.
\item At least one generic point of $(A^\Z,\sigma,\mu)$ is mean almost
periodic.
\item Every generic point of $(A^\Z,\sigma,\mu)$ is mean almost periodic.
\item At least one generic point of $(A^\Z,\sigma,\mu)$ is Besicovitch
almost periodic.
\item Almost every point of $(A^\Z,\sigma,\mu)$ is Besicovitch almost periodic.
\item $\cO(\tpi x)$ is relatively compact for at least one generic point $x$ of $(A^\Z,\sigma,\mu)$.
\item $\cO(\tpi x)$ is relatively compact for all generic points $x$ of $(A^\Z,\sigma,\mu)$.
\end{compactenum}
\end{theorem}

The next theorem adds additional information to items (ii) and (vi) of Theorem~\ref{theo:basic}. Its part b) is in line with our general attempt to study properties of individual orbits - not only of typical orbits (in a topological or measure theoretic sense).

\begin{theorem}\label{theo:precisions}
If $\DD(\mu)=0$ (equivalently if $(A^\Z,\sigma,\mu)$ has discrete spectrum), then
\begin{compactenum}[a)]
\item $\tpi:(A^\Z,\sigma,\mu)\to(\tX_\mu,\tsigma,\tilde\mu=\mu\circ\tpi^{-1})$ is an isomorphism of dynamical systems.
\item
The unique $D_1$-equivalence class $X_\mu\subseteq A^\Z$ with $\mu(X_\mu)=1$ is characterized as follows:
$$
X_\mu=\Bap(\mu):=\{x\in A^\Z:\text{ $x$ is $\mu$-generic and Besicovitch almost periodic}\}.
$$
\item For all $\epsilon>0$ and $x,y\in X_\mu$,
\begin{equation}
\mu(B_{d_1}(x,\epsilon))^{-1}\le\cN_\epsilon(\cO(x))=\cN_\epsilon(\cO(y))\le \mu(B_{d_1}(x,\epsilon/2))^{-1}.
\end{equation}
\end{compactenum}
\end{theorem}

\begin{proof}[Proof of Theorem~\ref{theo:basic}]
(The equivalence of
properties (ii) -- (vi) follows from Theorems 3.8 and 4.7 from \cite{Lenz2023}, but we prefer to give a self-contained proof for the whole theorem.\footnote{The equivalence of (ii) and (vi) is in reference \cite[Thm.~6]{Wiener-Wintner-2} by Wiener and Wintner. {That (ii) is equivalent to the mean almost periodicity of $\mu$-a.e.~point is stated in \cite[Lem.~5]{Vershik2011}}}).\\
First observe that $d_1(x,\sigma^jx)=\mu\{x\in A^\Z: x_0\neq x_j\}$ for each $\mu$-generic point $x$. Hence, if one such point is mean almost periodic, then all $\mu$-generic points are. Together with Lemma~\ref{lemma:compact-mean-a-p} this proves the equivalence of (iii), (iv), (vii), and (viii).\\[1mm]
(i) $\Rightarrow$ (iii):\;
Because of Proposition~\ref{proposition:R-D1}e), $(\tX_\mu,\td_1,\tsigma)$ is a minimal dynamical system. Hence each $\tx\in\tX_\mu$ is almost periodic w.r.t.~the metric $\td_1$, i.e.~each $x\in X_\mu$ is mean almost periodic. Since all points in $X_\mu$ are generic by Proposition~\ref{proposition:R-D1}b), this proves (iii).\footnote{This argument is similar to the proof of (i)$\Rightarrow$(ii) in Theorem~3.8 from \cite{Lenz2023}.}\\[1mm]
(vii) $\Rightarrow$ (ii):\; This follows from Lemma~\ref{lemma:compact-ONB}a.\\[1mm]
(ii) $\Rightarrow$ (i)\,\&\,(vi):\;
Denote by $\Eig(A^\Z,\sigma,\mu)$ the set of eigenvalues of the system and recall that in our case $\Eig(A^\Z,\sigma,\mu)=\{\xi_1,\xi_2,\xi_3,\dots\}$ is a countable subgroup of the torus  $\C_1:=\{z\in\C: |z|=1\}$. Denote their normalized eigenfunctions by $e_{\xi_1},e_{\xi_2},\dots$. (They are unique up to constant factors of modulus one.)
Let $f\in C(A^\Z,\C)$. As in the proof of
Theorem~4.7 of \cite{Lenz2023}  we consider the Fourier coefficients
$c_j=\langle f,e_{\xi_j}\rangle$ and observe that
for each $\epsilon>0$ there exists $k\in\N$ such that
\begin{equation}\label{eq:approximation}
\begin{split}
\limsup_{n\to\infty}\frac{1}{2n+1}\sum_{|i|\le n}
\left|f(\sigma^ix)-\sum_{j=1}^kc_j e_{\xi_j}(\sigma^ix)\right|
&=
\int_{A^\Z}\left|f-\sum_{j=1}^kc_j e_{\xi_j}\right|d\mu\\
&\le
\left\|f-\sum_{j=1}^kc_j e_{\xi_j}\right\|_2
<\epsilon
\end{split}
\end{equation}
for $\mu$-a.a.~points~$x$. This shows (vi).

Consider now the particular $f:A^\Z\to\C, x\mapsto x_0$. (We assume without loss that $A\subset\C$.)
Fix $x,y\in A^\Z$ such that \eqref{eq:approximation} and
\begin{equation}\label{eq:eigenfunctions}
e_{\xi_j}(\sigma^ix)=\xi_j^ie_{\xi_j}(x)\;(j\in\N,i\in\Z)
\end{equation}
holds for both of them.
(The latter may exclude a further set of measure zero.)
Choose $\ell\in\Z$ (depending on $x$ and $y$) such that $\sum_{j=1}^k|c_j|\cdot|e_{\xi_j}(y)/e_{\xi_j}(x)-\xi_j^\ell|<\epsilon$.\,\footnote{
Let $H:=\overline{\langle(\xi_1,\dots,\xi_k)\rangle}$. $H$ is a compact subgroup of $\C_1^k$, and the map $\phi:A^\Z\to\C_1^k/H,\ x\mapsto[(e_{\xi_1}(x),\dots,e_{\xi_k}(x))]_H$ is $\mu$-almost surely well defined. Observe that $\phi(\sigma x)=\phi(x)$.
As $\phi$ takes values in a standard Borel space, this can be treated like an $\sigma$-invariant real-valued measurable function. So $\phi$ is $\mu$-a.s.~constant,
i.e.~$(e_{\xi_1}(y)/e_{\xi_1}(x),\dots,e_{\xi_k}(x)/e_{\xi_k}(y))\in H$ for $\mu$-a.e.~$x$ and $y$.}
Hence, for $\mu$-a.a.~$x$ and $y$,
\begin{equation}\label{eq:synchronization}
\begin{split}
&D_1(x,y)
\le
d_1(x,\sigma^\ell y)
=
\limsup_{n\to\infty}\frac{1}{2n+1}\sum_{|i|\le n}\left|x_i-y_{i+\ell}\right|\\
&\le
\limsup_{n\to\infty}\frac{1}{2n+1}\sum_{|i|\le n}\left(
\left|f(\sigma^ix)-\sum_{j=1}^kc_j e_{\xi_j}(\sigma^ix)\right|
+\left|\sum_{j=1}^kc_j\xi_j^i(e_{\xi_j}(x)-\xi_j^{\ell}e_{\xi_j}(y))\right|\right.\\
&\hspace{3.3cm}\left.+\left|\sum_{j=1}^kc_j e_{\xi_j}(\sigma^{i+\ell}y)-f(\sigma^{i+\ell}y)\right|
\right)\\
&<
3\epsilon.
\end{split}
\end{equation}
It follows that $D_1(x,y)=0$ for $\mu$-a.a.~$x,y$, that is (i).\\[1mm]
(vi) $\Rightarrow$ (v) and (v) $\Rightarrow$ (iii) are obvious.
\end{proof}

\begin{proof}[Proof of Theorem~\ref{theo:precisions}]
a)
\;The sets $X_\mu$ and $\tX_\mu$ are well defined and $\mu(X_\mu)=1$ (Proposition~\ref{proposition:R-D1}).
The operator $U_\tpi:L^2_{\mu\circ\tpi^{-1}}\to L^2_\mu$, $\tilde f\mapsto \tilde f\circ\tpi$, is linear and  isometric by definition, and it is onto in view of Lemma~\ref{lemma:compact-ONB}b. (Observe Proposition~\ref{proposition:R-D1} to apply the lemma with $\tm=\mu\circ\tpi^{-1}$.) It follows that for each Borel set
$A\subseteq X_\mu$ there is a Borel set
$\tilde A\subseteq\tX_\mu$ such that $\mu(A\triangle\tpi^{-1}(\tilde A))=0$
(and vice versa, of course), so that the systems $(X_\mu,\sigma,\mu)$
and $(\tX,\tsigma,\mu\circ\tpi^{-1})$ are conjugate (via $\tpi$).
Then Theorem 2.6 of \cite{walters} combined with Theorem 2.2 [{\it ibid.}] shows that $\tpi$ is an (almost everywhere invertible) point isomorphism between these two systems.
\\[1mm]
b)
\;Let $x\in X_\mu$. Then $x$ is $\mu$-generic by Proposition~\ref{proposition:R-D1}b), and $x$ is Besicovitch almost periodic by Theorem~\ref{theo:basic}[(i)$\Rightarrow$(vi)] in conjunction with Lemma~\ref{lemma:D_1-Besicovitch}.
For the proof of the converse inclusion we recall some facts from \cite{Lenz2023}:
The limit
\begin{equation}\label{eq:A-def}
A(x,f,\bar\xi):=\lim_{n\to\infty}\frac{1}{2n+1}\sum_{|k|\le n}f(\sigma^kx)\bar\xi^k
\end{equation}
exists for all $f\in C(A^\Z,\C)$,  $x\in\Bap(\mu)$ and $\xi\in\C_1$. Define
$e_{f,\xi}:A^\Z\to\C$, $e_{f,\xi}(x)=A(x,f,\bar\xi)$ if $x\in\Bap(\mu)$ and $e_{f,\xi}(x)=0$ otherwise. Then $e_{f,\xi}(\sigma x)=\xi\,e_{f,\xi}(x)$ for all $x\in A^\Z$ and
$|e_{f,\xi}|$ is constant on $\Bap(\mu)$, call this constant $\gamma_{f,\xi}$. Moreover, $e_{f,\xi}$ is a pointwise defined version of $P_\xi f$, the orthogonal projection of $f$ to the (one-dimensional) eigenspace of $\xi$ in $L^2_\mu$ (which is trivial if $\xi\not\in\Eig:=\Eig(A^\Z,\sigma,\mu)$)
\cite[Thm.~4.7]{Lenz2023}\footnote{$A(x,f,\bar\xi)$ is denoted by $A(f_x\bar\xi)$ in the systematic treatment of \cite{Lenz2023}. For $\Z$-actions, as studied here, most of this is contained in \cite[Thms.~7 and~8]{Wiener-Wintner-2}.}. In particular, $f=\sum_{\,\xi\in \Eig}e_{f,\xi}$ in the sense of $L^2_\mu$-convergence, because $L^2_\mu$ has an ONB of eigenvectors, by assumption. More precisely, $f=\sum_{\,\xi\in E_f}e_{f,\xi}$, where $E_f=\{\xi\in\Eig: \gamma_{f,\xi}>0\}$.

Now fix one element $x\in\Bap(\mu)$. We will show that $D(x,y)=0$ for $\mu$-a.e.~$y\in X_\mu$. This implies $x\in X_\mu$ by Proposition~\ref{proposition:R-D1}b) and finishes the proof. We want to proceed as in \eqref{eq:synchronization}, but we need some changes. We choose $e_{\xi_j}=\gamma_{f,\xi_j}^{-1}e_{f,\xi_j}$
for the ONB used in \eqref{eq:approximation} and note that
$c_j=\langle f, e_{\xi_j}\rangle=\gamma_{f,\xi_j}^{-1}\langle f,P_\xi f\rangle
=\gamma_{f,\xi_j}^{-1}\|P_\xi f\|_2^2=\gamma_{f,\xi_j}$ for this choice.
Just as in~\eqref{eq:synchronization} we estimate for $\mu$-a.e.~$y$
\begin{equation*}
\begin{split}
D_1(x,y)
\le&
\limsup_{n\to\infty}\frac{1}{2n+1}\sum_{|i|\le n}
\left|f(\sigma^ix)-\sum_{j=1}^kc_j\xi_j^i e_{\xi_j}(x)\right|
+2\epsilon.
\end{split}
\end{equation*}
Instead of the remaining limsup (which depends on the particular fixed $x\in\Bap(\mu)$) we first estimate
\begin{equation*}
B:=\limsup_{n\to\infty}\frac{1}{2n+1}\sum_{|i|\le n}
\left|f(\sigma^ix)-\sum_{j=1}^kc_j\xi_j^i e_{\xi_j}(x)\right|^2.
\end{equation*}
Observing that $x\in\Bap(\mu)$ and that $f$ is continuous
so that
identity~\eqref{eq:A-def} applies, one equates
\begin{equation*}
\begin{split}
B
&=
\limsup_{n\to\infty}\frac{1}{2n+1}
\sum_{|i|\le n}
\left(f(\sigma^ix)-\sum_{j=1}^kc_j\xi_j^i e_{\xi_j}(x)\right)
\left(\bar f(\sigma^ix)-\sum_{j=1}^k\bar c_j \bar\xi_j^i\overline{e_{\xi_j}(x)}\right)\\
&=
\int|f|^2\,d\mu +\sum_{j=1}^k|c_j|^2-2\sum_{j=1}^k\Re\left(A(x,f,\bar\xi_j)\,\overline{c_j\,e_{\xi_j}(x)}\right)\\
&=
\int|f|^2\,d\mu -\sum_{j=1}^k|c_j|^2
=
\left\|f-\sum_{j=1}^kc_j e_{\xi_j}\right\|_2^2
<
\epsilon^2,
\end{split}
\end{equation*}
because $A(x,f,\bar\xi_j)=e_{f,\xi_j}(x)=c_j\,e_{\xi_j}(x)$.
This implies $D_1(x,y)\le 3\epsilon^{2/3}+2\epsilon$.\,\footnote{If $0\le a_i\in\R$ and $\limsup_{n\to\infty}\frac{1}{2n+1}\sum_{|i|\le n}a_i^2<\epsilon^2$, then $\limsup_{n\to\infty}\frac{1}{2n+1}\sum_{|i|\le n}a_i<3\epsilon^{2/3}$. For the proof subdivide the index set $\Z$ into $\{i:a_i\le \epsilon^{2/3}\}$, $\{i:\epsilon^{2/3}<a_i\le 1\}$ and $\{i:a_i>1\}$, and observe that the second set has upper density at most $\epsilon^{2/3}$.}\\[1mm]
c)\;
Let $\tx=\tpi x$ and $\ty=\tpi y$. Then $\tx,\ty\in\tX_\mu$. Because of Proposition~\ref{proposition:R-D1}e) there is an isometry on $\tX_\mu$ (a translation) that maps $\cO(\tx)$ onto $\cO(\ty)$ and leaves the measure $\tilde\mu$ invariant. Hence $\cN_\epsilon(\cO(x))=\cN_\epsilon(\cO(\tx))=\cN_\epsilon(\cO(\ty))=\cN_\epsilon(\cO(y))$.

For a lower bound of
$\cN_\epsilon(\cO(\tx))$
we follow the argument in the proof of the elementary part of Frostman's lemma:
Suppose that $B_{\td_1}(\tx_1,\epsilon),\dots, B_{\td_1}(\tx_N,\epsilon)$ with $\tx_i\in\cO(\tx)$ is a covering of $\cO(\tx)$. Let $M>1$. Then
$B_{\td_1}(\tx_1,M\epsilon),\dots, B_{\td_1}(\tx_N,M\epsilon)$ is a covering
of $\overline{\cO(\tx)}$ satisfying
\begin{equation*}
1=\tilde\mu(\overline{\cO(\tx)})\le \sum_{i=1}^N\tilde\mu(B_{\td_1}(\tx_i,M\epsilon))\le N\,f(M\epsilon)
\end{equation*}
where $f(\epsilon):=\sup\{\tilde\mu(B_{\td_1}(\tS^n\tx,\epsilon)): n\in\Z\}$,
so that $\cN_\epsilon(\cO(\tx)) \ge 1/f(M\epsilon)$. As $\tS$ leaves the metric $\td_1$ and the measure $\tilde\mu$ invariant, we have in view of
Lemma~\ref{lemma:d1-formula}:
\begin{equation*}
\begin{split}
f(M\epsilon)
&=
\sup\{\tilde\mu(\tS^n(B_{\td_1}(\tx,M\epsilon))): n\in\Z\}=
\tilde\mu(B_{\td_1}(\tx,M\epsilon))
\searrow
\tilde\mu(B_{\td_1}(\tx,\epsilon))\\
&=\mu(B_{d_1}(x,\epsilon))
\end{split}
\end{equation*}
as $M\searrow 1$.
For an upper bound  of
$\cN_\epsilon(\cO(\tx))$ we
suppose that
$\{\tx_1,\dots,\tx_N\}$ is a maximal $\epsilon$-separated subset of $\cO(\tx)$. Then
$B_{d_1}(\tx_1,\epsilon),\dots,$ $ B_{d_1}(\tx_N,\epsilon)$ is a covering of $\cO(\tx)$, so that $\cN_\epsilon(\cO(\tx))\le N$, and $N\le 1/f(\epsilon/2)$ in view of Lemma~\ref{lemma:relatively-compact}.
\end{proof}

\begin{corollary}
If $X\subset A^\Z$ is a minimal closed subshift and if $(X,\sigma)$ is mean equicontinuous (see Subsection~\ref{subsec:Bcn}) with unique invariant measure $\mu$, then $X_\mu=X$.
\end{corollary}
\begin{proof}
Let $x,y\in X$. As $(X,\sigma)$ is minimal, $y\in\overline{\cO(x)}$. In view of the mean equicontinuity this implies $D_1(x,y)=\inf_{n\in\Z}d_1(\sigma^nx,y)=0$. Hence all $x,y\in X$ belong to the same $D_1$-equivalence class, so that $X\subseteq X_\mu\subseteq X$.
\end{proof}

\begin{remark}[Weakly mixing systems]\label{remark:weak-mixing}
If $(A^\Z,\sigma,\mu)$ is weakly mixing, the invariant {pseudometric} $d_1$ equals a constant $\dd(\mu)\ge0$ $\mu\times\mu$-a.e.\footnote{This observation is contained in \cite[Thm.~3]{Vershik2011}, which shows much more: $(\sigma,\mu)$ is weakly mixing if and only if the {pseudometric} $\limsup_{n\to\infty}\frac{1}{n}\sum_{k=0}^{n-1}\rho(\sigma^kx,\sigma^ky)$ is constant $\mu\times\mu$-a.e.~for every cut {pseudometric} $\rho$ on $A^\Z$, where a cut {pseudometric} $\rho$ is determined by a finite measurable partition of $A^\Z$: $\rho(x,y)=1$ if $x$ and $y$ belong to different elements of the partition, and $\rho(x,y)=0$ otherwise.} If $\mu$ is not just a one-point measure, then $\dd(\mu)>0$.\footnote{Suppose for a contradiction that $\dd(\mu)=0$. Denote by $X_0$ the set of all $x\in A^\Z$ such that $d_1(x,y)=0$ for $\mu$-a.a.~$y$. Then $\mu(X_0)=1$, and if $x,x'$ are points in $X_0$, there exists $y\in A^\Z$ such that $d_1(x,y)=0=d_1(x',y)$, whence $d_1(x,x')=0$. Let $X_0':=X_0\cap \sigma^{-1}(X_0)\cap\{\text{generic points of }(\sigma,\mu)\}$. Then $\mu(X_0')=1$ and $d_1(x,\sigma x)=0$ for each $x\in X_0'$. This implies that all blocks $ab$ with $a,b\in A$, $a\neq b$, occur with density zero in $x$, and as $x$ is generic, these blocks have probability zero under $\mu$, so that the ergodic measure $\mu$ must be a one-point mass.
For example, if $A=\{0,1\}$ and $\mu$ is the $(p,1-p)$-Bernoulli measure, then $\dd(\mu)=2p(1-p)$.}
As $\mu\times \mu$ is invariant under all $\id\times \sigma^n$, $n\in\Z$, also
$D_1(x,y)=\dd(\mu)$ for $\mu\times\mu$-a.a.~$(x,y)$, i.e.~$x$ and $y$ belong to different $D_1$-equivalence classes for almost all pairs $(x,y)$. This is in maximal contrast to what happens in the case of discrete spectrum where, according to Theorem~\ref{theo:basic}, $x$ and $y$ belong to the same $D_1$-equivalence class for almost all pairs $(x,y)$.

Suppose now that $B$ is another alphabet and $F:A^\Z\to B^\Z$ is a nontrivial continuous factor map. Then also $(B^\Z,\sigma,\mu\circ F^{-1})$ is weakly mixing, so $\dd(\mu\circ F^{-1})>0$, and, for $\mu$-a.e.~$x$, $F(x)$ is not mean almost periodic by Theorem~\ref{theo:basic}. Vershik \cite[Thm.~5]{Vershik2011} proves that this property actually characterizes weakly mixing systems $(A^\Z,\sigma,\mu)$.
\end{remark}

\subsection{The case of weak model subsets of $\Z$ with Borel window}

Let $G$ be a compact, metric, abelian group with normalized Haar measure $\lambda$ and $\Delta:\Z\to G$ an injective group homomorphism with $\overline{\Delta(\Z)}=G$.
Denote by $R:G\to G,\ g\mapsto g+\Delta(1)$ and note that $\Delta(n+1)=R(\Delta(n))$, so that $\Delta(\Z)$ is the $R$-orbit of $0\in G$. Fix a Borel subset $W\subset G$ with $0<\lambda(W)<1$, called the ``window''. Define $\varphi:G\to A^\Z$, $\varphi(g)=(1_W(R^kg))_{k\in\Z}$. Then
$\mu:=\lambda\circ\varphi^{-1}$ is the \emph{Mirsky measure} on $\{0,1\}^\Z$ determined by $\Delta$ and $W$. As a factor of $(G,R,\lambda)$ the system $(\{0,1\}^\Z,\sigma,\mu)$ has discrete spectrum.
Theorem~\ref{theo:basic} and Proposition~\ref{proposition:R-D1} imply:

\begin{corollary}\label{coro:basic-model-set}
If $\mu$ is the Mirsky measure determined by $\Delta$ and $W$, then $\DD(\mu)=0$. In particular,
$(\tX_\mu,\tsigma)$ is a minimal, isometric, uniquely ergodic dynamical system with invariant measure $\tilde\mu:=\mu\circ\tpi^{-1}$, and the system $(A^\Z,\sigma,\mu)$ is isomorphic to the system $(\tX_\mu,\tsigma,\tilde\mu)$.
\end{corollary}

Denote
\begin{equation*}
G_{gen}:=\{g\in G:\varphi(g)\text{ is $\mu$-generic}\}.
\end{equation*}
\begin{lemma}\label{lemma:d1-formula}
\begin{compactenum}[a)]
\item
There is a measurable set $G_0\subseteq G$ with $\lambda(G_0)=1$ such that, for all $g,g'\in G_0$,
\begin{equation}\label{eq:d1-formula}
d_1(\varphi(g),\varphi(g'))=d_W(g,g'):=\lambda((W-g')\triangle(W-g)).
\end{equation}
(Observe that $d_W$ is a pseudo-metric on $G$.)
\item
If the window $W\subset G$ is compact, then the set $G_0$ from item a) can be chosen to be $G_{gen}$.
\end{compactenum}
\end{lemma}
\begin{proof}
a)
\;Let $\varepsilon>0$. Since $G$ is a compact metrizable space, there there is a Borel set $W_\varepsilon\subseteq G$ with $\lambda(\partial W_\varepsilon)=0$ and $\lambda(W\triangle W_\varepsilon)<\varepsilon$. Denote by $\varphi_\varepsilon$ the coding map associated with the window $W_\varepsilon$.
Then, for any $g,g'\in G$,
\begin{equation*}
\begin{split}
d_1(\varphi_\varepsilon(g'),\varphi_\varepsilon(g))
&=
\du\{k\in\Z:\ 1_{W_\varepsilon}(R^kg')\neq 1_{W_\varepsilon}(R^kg)\}\\
&=
\limsup_{n\to\infty}\frac{1}{2n+1}\sum_{|k|\le n}1_{{W_\varepsilon}\triangle (W_\varepsilon+(g'-g))}(R^kg)\\
&=
\lambda({W_\varepsilon}\triangle (W_\varepsilon+(g'-g)))
=
\lambda((W_\varepsilon-g')\triangle (W_\varepsilon-g))
\end{split}
\end{equation*}
by the uniform ergodic theorem, because $\lambda(\partial({W_\varepsilon}\triangle (W_\varepsilon+(g-g'))))\le\lambda(\partial W_\varepsilon)+\lambda(\partial(W_\varepsilon+(g-g')))=0$ for all $g,g'\in G$.
Hence
\begin{equation}\label{eq:d_1-estimate for later}
\begin{split}
&\ |d_1(\varphi(g'),\varphi(g))-\lambda((W-g')\triangle(W-g))|\\
\le &\
|d_1(\varphi(g'),\varphi(g))-d_1(\varphi_\varepsilon(g'),\varphi_\varepsilon(g))|\\
&\hspace*{12mm}+
|\lambda((W_\varepsilon-g')\triangle(W_\varepsilon-g))-\lambda((W-g')\triangle(W-g))|\\
\le &\
d_1(\varphi(g'),\varphi_\varepsilon(g'))+d_1(\varphi(g),\varphi_\varepsilon(g))
+
2\lambda(W_\varepsilon\triangle W)\\
\le &\
\limsup_{n\to\infty}\frac{1}{2n+1}\sum_{|k|\le n}1_{W_\varepsilon\triangle W}(R^kg')
+
\limsup_{n\to\infty}\frac{1}{2n+1}\sum_{|k|\le n}1_{W_\varepsilon\triangle W}(R^kg)
+
2\varepsilon.
\end{split}
\end{equation}
By Birkhoff's ergodic theorem there is a set $G_\varepsilon$ with $\lambda(G_\varepsilon)=1$ such that for all $g,g'\in G_\varepsilon$ the two $\limsup$s equal $\lambda(W_\varepsilon\triangle W)$, whence
$|d_1(\varphi(g'),\varphi(g))-\lambda((W-g')\triangle(W-g))|<4\varepsilon$ for all $g,g'\in G_\varepsilon$. Letting $\varepsilon=\varepsilon_n\to0$ proves assertion a).\\[2mm]
b)
\;The proof proceeds as in a), except that the use of Birkhoff's theorem in~\eqref{eq:d_1-estimate for later} must be avoided:
Fix a metric for the topology on $G$ and
choose $W_\varepsilon$ as an open $r_\varepsilon$-neighborhood of $W$ such that $\lambda(W_\varepsilon\setminus W)<\varepsilon$. The choice of the $r_\varepsilon$ can be made such that $\lambda(\partial W_\varepsilon)=0$ for all $\varepsilon>0$. Hence, for all $g\in G_{gen}$,
\begin{equation*}
\begin{split}
&\limsup_{n\to\infty}\frac{1}{2n+1}\sum_{|k|\le n}1_{W_\varepsilon\triangle W}(R^kg)\\
&=
\limsup_{n\to\infty}\frac{1}{2n+1}\left(\sum_{|k|\le n}1_{W_\varepsilon}(R^kg)-
\sum_{|k|\le n}1_{W}(R^kg)\right)\\
&=
\lambda(W_\varepsilon)-\liminf_{n\to\infty}\frac{1}{2n+1}\sum_{|k|\le n}(\varphi(g))_k\\
&=
\lambda(W_\varepsilon)-\mu([1])
=
\lambda(W_\varepsilon)-\lambda(W)<\varepsilon.
\end{split}
\end{equation*}
\end{proof}

\begin{remark}\label{remark:G_0=G?}
\textbf{($\tpi\circ\varphi$ is an ``almost'' automorphism from $G$ to $\tX_\mu$)}\\
Recall from \cite{KR2018} that $W$ is \emph{Haar aperiodic}, if $\lambda(W\triangle(W+h))=0$ implies $h=0$ ($h\in G$).
In this case $d_W$ from \eqref{eq:d1-formula} is a translation invariant metric on $G$ which is continuous for the topology on $G$, and as $G$ is compact, $d_W$ generates the topology on $G$.
Denote $G_1=G_0\cap \varphi^{-1}(X_\mu)$ with $G_0$ from Lemma~\ref{lemma:d1-formula}.
Then $\lambda(G_1)=1$ and $\tpi\circ\varphi_{|G_1}:G_1\to \tX_\mu$ is isometric and extends to an isometry $\Phi$ from $G$ to $\tX_\mu$.
It is easily seen that $\Phi$ is also a group isomorphism.

Hence one may choose
$G_0=G_1=\{g\in G: \tpi\circ\varphi(g)=\Phi(g)\}$: indeed, if $g,g'$ are from this set, then $\tpi\circ\varphi(g)=\Phi(g)\in\tX_\mu$ so that $\varphi(g)\in X_\mu$,
see Proposition~\ref{proposition:R-D1}d, and $d_1(\varphi(g),\varphi(g'))=\td_1(\tpi\circ\varphi(g),\tpi\circ\varphi(g'))=\td_1(\Phi(g),\Phi(g'))=d_W(g,g')$.
 Observe that, for this (maximal) choice of $G_0$, $\tpi\circ\varphi(G_0)=\Phi(G_0)\subseteq\tX_\mu$ with equality if and only if $G_0=G$. The example in Section~\ref{sec:rotation} can be produced in such a way that $G_0\neq G$, see Footnote~\ref{foot:arbitrary-choice}.

Note that $\widetilde W:=\Phi(W)$ is a Borel window
in $\tX_\mu$. It is compact if and only if $W$ is.
The corresponding coding map $\tilde\varphi:\tX_\mu\to\{0,1\}^\Z$ satisfies
$\tilde\varphi(\Phi(g))=\varphi(g)$ for all $g$ such that $\tpi\circ\varphi(g)=\Phi(g)$, in particular for $\mu$-a.e. $g$.
\end{remark}
\begin{remark}[Ergodic measure on $\{0,1\}^\Z$ with pure point spectrum are Mirsky measures]\quad\\
If $(\{0,1\}^\Z,\sigma,\mu)$ is any ergodic system with discrete spectrum, we know from Theorem~\ref{theo:precisions}a that there are a $\sigma$-invariant Borel set $Y\subseteq A^\Z$ and a $\tsigma$-invariant Borel set $\widetilde Y\subseteq\tX_\mu$ with $\mu(Y)=\tilde\mu(\widetilde Y)=1$ such that $\tpi_{|Y}:Y\to\widetilde Y$ is a bimeasurable bijection satisfying $\tsigma\circ\tpi=\tpi\circ\sigma$ on $Y$. Then $\widetilde W:=\tpi(Y\cap[1]_0)$ is a Borel window in $\widetilde Y\subset\tX_\mu$. If $\tilde\varphi$ denotes the coding w.r.t.~this window, then $\tilde\varphi(\tpi(x))=x$ for all $x\in Y$, hence for $\mu$-a.e.~$x$. In particular, $\mu=\tilde\mu\circ\tilde\varphi^{-1}$ is the Mirsky measure generated by $(\tX_\mu,\tsigma,\tilde\mu)$ and $\widetilde W$.
\end{remark}

\begin{remark}\label{remark:B-R(mu)=0}
Let $x=\varphi(g)$, $x'=\varphi(g')$. As $\sigma^nx'=\varphi(R^ng')$ for all $x'=\varphi(g')$ and $n\in\Z$, Lemma~\ref{lemma:d1-formula} gives an alternative proof of the fact $\DD(\mu)=\DD(\lambda\circ\varphi^{-1})=0$:
\begin{equation}\label{eq:D_1-approximation}
D_1(x,x')=\inf\{d_1(x,\sigma^nx'):n\in\Z\}
=
\inf_{n\in\Z}\lambda(W\triangle((W-(R^ng'-g)))=0
\end{equation}
for $\mu\times\mu$-a.a.~$(x,x')$, because the $R$-orbit of $g'$ approaches $g$ arbitrarily close.
\end{remark}

\section{Specializing to $\cB$-free systems}\label{sec:B-free}

\subsection{Some consequences of the general theory}\label{subsec:B-free-intro}

For a primitive\,\footnote{$\cB\subseteq\N\setminus\{1\}$ is primitive, if no element of $\cB$ divides any other element of $\cB$.} set $\cB\subseteq\N\setminus\{1\}$ let $\cM_\cB=\bigcup_{b\in\cB}b\Z$ be its set of multiples and $\cF_\cB=\Z\setminus\cM_\cB$ the set of $\cB$-free numbers. Let $\eta=\eta_\cB:=1_{\cF_\cB}$. We assume that $\cB$ is a \emph{Besicovitch set}, i.e.~the
right one-sided density $d_\N(\cB):=\lim_{n\to\infty}\frac 1{n}\card(\cM_\cB\cap\{1,\dots,n\})$ exists. Since $\cM_\cB=-\cM_\cB$ this is equivalent to the existence of the density
$d(\cM_\cB):=\lim_{n\to\infty}\frac 1{2n+1}\card(\cM_\cB\cap\{-n,\dots,n\})$. In the sequel we will use the notations $d,d_\N$ for such densities and $\du,\du_\N$ for upper densities, correspondingly. By $d_1$ we continue to denote the Besicovitch pseudo-metric.

We recall the setting from \cite{KKL2016} (with minor notational changes):
Let $H:=\prod_{b\in\cB}\Z/b\Z$, $\Delta:\Z\to H,\ n\mapsto(n\mod b)_{b\in\cB}$ the canonical embedding, and $G:=\overline{\Delta(\Z)}$. $G$ is the \emph{odometer group} associated with $\cB$, and we denote its normalized \emph{Haar measure} by $\lambda$. The compact set $W:=\{g\in G:g_b\neq0\ (\forall b\in\cB)\}$ is the \emph{window} associated with $\cB$. Denote by $R:G\to G,\ g\mapsto g+\Delta(1)$ and note that $\Delta(n+1)=R(\Delta(n))$, so that $\Delta(\Z)$ is the $R$-orbit of $0\in G$.
Define $\varphi:G\to A^\Z$, $\varphi(g)=(1_W(R^kg))_{k\in\Z}$ and observe that $\eta=\varphi(\Delta(0))$. Then
$\mu:=\lambda\circ\varphi^{-1}$ is the \emph{Mirsky measure} on $\{0,1\}^\Z$ determined by $\Delta$ and $W$.
In this situation Corollary~\ref{coro:basic-model-set} and Lemma~\ref{lemma:d1-formula}b apply.
In particular:

\begin{lemma}\label{lemma:d1-formula-B-free}
Suppose $\cB$ is a Besicovitch set. Then
\begin{align}
d_1(\eta,\varphi(g))\label{eq:d1-formula-2}
&=
\lambda(W\triangle(W-g))\quad\text{for all $g$ for which $\varphi(g)$ is $\mu$-generic},\\
d_1(\eta,\sigma^r\eta)\label{eq:d1-formula-3}
&=
\lambda(W\triangle(W-\Delta(r)))=d(\cM_\cB\triangle(\cM_\cB-r))\quad\text{for all $r\in\Z$.}
\end{align}
\end{lemma}

\begin{proof}
If $\cB$ is a Besicovitch set, then $\eta=\varphi(\Delta(0))$ is $\mu$-generic (as a consequence of the Erd\H{o}s-Davenport theorem, see~\cite{BKKL2015}), and so is $\sigma^r\eta$. Hence \eqref{eq:d1-formula-2} and the first identity from \eqref{eq:d1-formula-3} follow from
Lemma~\ref{lemma:d1-formula}b.
Finally,
$d_1(\eta,\sigma^r\eta)=\du(\cM_\cB\triangle(\cM_\cB-r))$ by rewriting the definitions, and this equals $d(\cM_\cB\triangle(\cM_\cB-r))$ by Corollary~\ref{coro:density-exists}.
\end{proof}

\begin{corollary}\label{coro:eta-in-X_mu}
If $\cB$ is a Besicovitch set, then $\eta\in X_\mu$, where $\mu$ is the Mirsky measure. Moreover, $\eta$ is Besicovitch almost periodic, and $\overline{\cO(\tilde{\eta})}=\tX_\mu$.\footnote{Attention: In many papers on $\cB$-free systems $\tX$ denotes the hereditary closure of a subshift~$X\subseteq\{0,1\}^\Z$.}
\end{corollary}

\begin{proof}
In view of \eqref{eq:d1-formula-2}, identity~\eqref{eq:D_1-approximation} from Remark~\ref{remark:B-R(mu)=0} also holds for $x'=\eta$. Hence
$D_1(\eta,x)=0$ for $\mu$-a.e.~$x$ so that
$\eta\in X_\mu$, Theorem~\ref{theo:precisions}b implies that $\eta$ is Besicovitch almost periodic, and $\overline{\cO(\tilde{\eta})}=\tX_\mu$ follows from Proposition~\ref{proposition:R-D1}d. Alternatively one can use the characterization of $X_\mu$ from Theorem~\ref{theo:precisions}b, which applies to $\eta$ in view of \cite[Cor.~2.16]{Bergelson2019} or (the much more general) Proposition 3.39 and Remark 3.41 from \cite{Lenz2020}.
\end{proof}

In the following corollary we consider taut sets: The set $\cB$ is \emph{taut} \cite{Hall1996}, if $\ddelta(\cM_{\cB\setminus\{b\}})<\ddelta(\cM_\cB)$ for all $b\in\cB$, see \cite{BKKL2015} and \cite{KKL2016} for a discussion of this concept in a dynamical context. In particular: if $\cB$ is taut, then its associated window is Haar aperiodic \cite[Rem.~1.2]{KKL2016}. So the next corollary is just a special case of Remark~\ref{remark:B-R(mu)=0}.

\begin{corollary}
If $\cB$ is a taut Besicovitch set, then
$\tpi\circ\varphi:(G,R,\lambda)\to (\tX_\mu,\sigma,\mu\circ\tpi^{-1})$ is an isomorphism.
\end{corollary}

\subsection{Arithmetic identities and estimates for $d_1(\eta,\sigma^r\eta)$}\label{subsec:arithmetic}

In this subsection we always assume that $\eta=\eta_\cB$ for a Besicovitch set $\cB$.
Recall from Lemma~\ref{lemma:d1-formula-B-free} (see also Corollary~\ref{coro:density-exists}) that, for all $r\in\Z$,
\begin{equation}\label{eq:distance-density}
d_1(\eta,\sigma^r\eta)
=
d(\cM_\cB\triangle(r+\cM_\cB))
=
2\left(d(\cM_\cB\cup(r+\cM_\cB))-d(\cM_\cB)\right).
\end{equation}

\begin{proposition}\label{prop:like-Rogers}
If $r\mid r'$, then $d(\cM_\cB\cup(r+\cM_\cB))\ge d(\cM_\cB\cup(r'+\cM_\cB))$ and
$d_1(\eta,\sigma^r\eta)\ge d_1(\eta,\sigma^{r'}\eta)$.
\end{proposition}
\noindent The proof of this Proposition is deferred to Subsection~\ref{subsec:Proof-of-Prop-Rogers}. It adapts the strategy of the proof of Rogers' Theorem as presented in \cite[Ch.~V, \S6]{Halberstam-Roth}.

Fix $r\in\Z$ and denote
\begin{equation*}
\Br=\{b\in\cB:b\mid r\}.
\end{equation*} For $b\in\cB$ let
\begin{equation*}
Q_b=b\Z\cup(r+b\Z).
\end{equation*}
We are going to estimate the density of $\cM_\cB\cup(r+\cM_\cB)=\bigcup_{b\in \cB}Q_b$.
For a finite subset $S\subseteq \cB$ set
\begin{equation*}
P_S=\bigcap_{b\in S}b\Z=\lcm(S)\cdot\Z.
\end{equation*}
The Chinese Remainder Theorem implies:
\begin{lemma}\label{lem:intersection}
$P_K\cap (r+P_{S\setminus K})\neq\emptyset$ if and only if
$\gcd(\lcm(K),\lcm(S\setminus K))\mid r$ (for finite $K\subseteq S\subset\cB$),
and in this case $d(P_K\cap(r+P_{S\setminus K}))=1/\lcm(S)$.
\end{lemma}
\begin{lemma}\label{lem:disjointP}
Let $r\in\Z$ and let  $K,K'\subseteq S$ be finite subsets of $\cB$. The following are equivalent:
\begin{compactenum}[(a)]
\item $P_K\cap (r+P_{S\setminus K})\cap P_{K'}\cap (r+P_{S\setminus K'})\neq\emptyset$,
\item $K\setminus \Br=K'\setminus \Br$ and $\gcd(\lcm(K),\lcm(S\setminus K))\mid r$,
\item $P_K\cap (r+P_{S\setminus K})= P_{K'}\cap (r+P_{S\setminus K'})\neq\emptyset$
\end{compactenum}
\end{lemma}
\begin{proof} (a)$\Rightarrow$(b)\; Since $P_K\cap (r+P_{S\setminus K})\neq\emptyset$, we have $\gcd(\lcm(K),\lcm(S\setminus K))|r$ by Lemma~\ref{lem:intersection}.
Let $b\in K\setminus \Br$ and suppose  $b\in S\setminus K'$. Let $x\in P_K\cap
(r+P_{S\setminus K'})$. Then
$x=0 \mod b$ and $x=r\mod b$, hence $b\mid r$ in contradiction to $b\notin \Br$. Similarly, we prove that $K'\setminus \Br\subseteq K$.\\
(b)$\Rightarrow$(c)\; Since $\gcd(\lcm(K),\lcm(S\setminus K))\mid r$,
we have $P_K\cap (r+P_{S\setminus K})\neq\emptyset$ by Lemma~\ref{lem:intersection}.
Let $x\in P_K\cap (r+P_{S\setminus K})$ and $b\in K'$. If $b\in K$, then $b\mid x$. Otherwise $b\in S\setminus K$, so $b\mid x-r$, and $b\in \Br$ (by (b)), so  $b\mid x$ as well. We have proved that $x\in P_{K'}$. Similarly, we prove that $x\in r+P_{S\setminus K'}$. Therefore,
 $P_K\cap (r+P_{S\setminus K})\subseteq P_{K'}\cap (r+P_{S\setminus K'})$. The other inclusion follows analogously.\\
(c)$\Rightarrow$(a)\; is obvious.
\end{proof}

Given $r\in \Z$ and a finite $S\subseteq \cB$, let $T_r(S)$ denote the number of subsets $K\subseteq S\setminus \Br$ such that $\gcd(\lcm(K),\lcm(S\setminus K))|r$.

\begin{remark}\label{rem:Tr(S)}
If $\cB$ is pairwise coprime, then $T_r(S)=2^{|S\setminus \Br|}$ for all $r\in\Z$.
\end{remark}

\begin{lemma}\label{lem:sum}
Assume that $r\in \Z$. The density of the set $\cM_\cB\cup(r+\cM_\cB)=\bigcup_{b\in \cB}Q_b$ equals
\begin{equation*}
\sum_{\emptyset\neq S\subseteq\cB:|S|<\infty}(-1)^{|S|+1}\frac{T_r(S)}{\lcm(S)},
\end{equation*}
which, if not absolutely convergent, is interpreted as the (monotone!) limit
\begin{equation*}
\lim_{n\to\infty}\sum_{\emptyset\neq S\subseteq\cB\cap\{1,\dots,n\}}(-1)^{|S|+1}\frac{T_r(S)}{\lcm(S)}
\end{equation*}
\end{lemma}

\begin{proof} (For finite $\cB$, in the general situation we pass to the limit).
The density of the set $\bigcup_{b\in \cB}Q_b$ is the sum of the expressions
\begin{equation*}
(-1)^{k+1}\sum_{\emptyset\neq S\subseteq \cB:|S|=k}d\left(\bigcap_{b\in S}Q_b\right),
\end{equation*}
where $k$ runs through {$\{1,\dots,\card\cB\}$}. So it is enough to prove that
\begin{equation}\label{eq:QandP}
d\left(\bigcap_{b\in S}Q_b\right)=\frac{T_r(S)}{\lcm(S)}.
\end{equation}
Note that
\begin{equation*}
\bigcap_{b\in S}Q_b=\bigcup_{K\subseteq S}P_K\cap (r+P_{S\setminus K})= \bigcup_{K\subseteq S\setminus \Br}P_K\cap (r+P_{S\setminus K});
\end{equation*}
the r.h.s equation follows from Lemma \ref{lem:disjointP} (applied with $K'=K\setminus \Br$).
By Lemma \ref{lem:intersection},
he density of each non-empty set $P_K\cap (r+P_{S\setminus K})$ equals $\frac{1}{\lcm(S)}$, and,
by Lemma \ref{lem:disjointP}, the sets $P_K\cap (r+P_{S\setminus K})$ are disjoint for different $K\subseteq S\setminus \Br$.
Hence (\ref{eq:QandP}) follows.
\end{proof}

\begin{corollary}\label{cor:indif}
If $r'=rm$ and $m$ is coprime to every element of $\cB$, then
$d_1(\eta,S^r\eta)=d_1(\eta,S^{r'}\eta)$.
\end{corollary}

\begin{proof}
The equivalent assertion $d(\cM_{\cB}\cup(\cM_{\cB}+r))=d(\cM_{\cB}\cup(\cM_{\cB}+r'))$ (see \eqref{eq:distance-density}) follows from Lemma~\ref{lem:sum}, because, under the assumptions of the corollary, $\gcd(\lcm(K),\lcm(S\setminus K))|r$ if and only if $\gcd(\lcm(K),\lcm(S\setminus K))|r'$.
\end{proof}

\begin{corollary}\label{cor:bar}
If $\cB$ is pairwise coprime and  $0\neq r\in\Z$, then
$d_1(\eta,\sigma^r\eta)=d_1(\eta,\sigma^{r'}\eta)$,
where $r':=\lcm(\Br)$.
\end{corollary}
\begin{proof}
We shall prove the equivalent statement
$d(\cM_{\cB}\cup(\cM_{\cB}+r))=d(\cM_{\cB}\cup(\cM_{\cB}+r'))$, see \eqref{eq:distance-density}.
Note that
$\Br=\{b\in \cB:b\mid r\}=\{b\in \cB:b\mid r'\}$.
The assertion follows by  Lemma \ref{lem:intersection} and Lemma \ref{lem:sum} since $\gcd(\lcm(K),\lcm(S\setminus K))=1$ for every $K\subseteq S$.
\end{proof}

\begin{corollary}\label{cor:formula-G}
If the elements of $\cB$ are pairwise coprime, then
\begin{equation}\label{eq:coprime-formula}
d\left(\cM_\cB\cup(r+\cM_\cB)\right)
=
1-d(\cF_\cB)\cdot\prod_{b\in \cB\setminus \Br}\left(1-\frac{1}{b-1}\right)
\end{equation}
and
\begin{equation}\label{eq:coprime-distance}
d_1(\eta,\sigma^r\eta)
=
2\,d(\cF_\cB)\cdot\left(1-\prod_{b\in \cB\setminus \Br}\left(1-\frac{1}{b-1}\right)\right)
\end{equation}
for every $r\in\Z$.
\end{corollary}

\begin{proof}
Denote $\cB_n=\cB\cap\{1,\ldots,n\}$ and assume that $\Br\subseteq\cB_n$.
By Lemma~\ref{lem:sum} and Remark~\ref{rem:Tr(S)},
\begin{equation*}
\begin{split}
&d(\cM_{\cB_n}\cup(r+\cM_{\cB_n}))\\
&=
\sum_{\emptyset\neq S\subseteq\cB_n}(-1)^{|S|+1}\frac{2^{|S\setminus \Br|}}{\lcm(S)}\\
&=
\sum_{\emptyset\neq S\subseteq \Br}(-1)^{|S|+1}\frac{1}{\lcm(S)}\\
&\qquad+
\sum_{S'\subseteq \Br}\hspace*{0.5cm}\sum_{\emptyset\neq S\subseteq\cB_n\setminus \Br}(-1)^{|S|+|S'|+1}\frac{2^{|S|}}{\lcm(S')\cdot\lcm(S)}\\
&=
1-\prod_{b\in \Br}\left(1-\frac{1}{b}\right)
+\prod_{b\in \Br}\left(1-\frac{1}{b}\right)\cdot\left(1-\prod_{b\in\cB_n\setminus \Br}\left(1-\frac{2}{b}\right)\right)\\
&=
1-d(\cF_{\cB_n})\cdot\prod_{b\in \cB_n\setminus \Br}\left(1-\frac{1}{b-1}\right).
\end{split}
\end{equation*}
In the limit $n\to\infty$ this yields \eqref{eq:coprime-formula}, and \eqref{eq:coprime-distance} follows now from \eqref{eq:distance-density}.
\end{proof}

\begin{corollary}\label{coro:Erdos-distances}
Let $\cB=\{b_i:i>0\}$ where $b_1<b_2<\dots$ are pairwise coprime,
$S_n=\{b_1,\dots,b_n\}$, $\ell_n=\lcm(S_n)$, and $\eta=\eta_\cB$. Then
\begin{equation*}
\begin{split}
\min_{1\le r\le\ell_n}\frac{d_1(\eta,S^r\eta)}{2\,d(\cF_\cB)}
&=
1-\max_{1\le r\le\ell_n}\prod_{b_i\not\mid r}\left(1-\frac{1}{b_i-1}\right)\\
&\ge
1-\prod_{b_i\not\mid\ell_n}\left(1-\frac{1}{b_i-1}\right)
=
\frac{d_1(\eta,S^{\ell_n}\eta)}{2\,d(\cF_\cB)}.
\end{split}
\end{equation*}
\end{corollary}
\begin{proof}
The identities follow from Corollary~\ref{cor:formula-G},
and the inequality is due to the following observation: As $r<\ell_n$, there are at most $n$ numbers $b'_i$ for which $b'_i\mid r$. Hence
\begin{equation*}
\prod_{b_i\mid\ell_n}\left(1-\frac{1}{b_i-1}\right)
\le
\prod_{b'_i\mid r}\left(1-\frac{1}{b'_i-1}\right).
\end{equation*}
\end{proof}

\subsection{Estimates for $\cN_\epsilon(\cO(\tilde{\eta}))$}\label{subsec:covering}

For $r\in\Z$ let
$\epsilon_r=d_1(\eta,\sigma^{r}\eta)$, and for finite $S\subset\cB$ denote $\ell_S=\lcm(S)$ and $\epsilon_S=\epsilon_{\ell_S}$.

\begin{lemma}[Upper estimate]\label{lemma:upper-estimate}
Assume that $\cB$ is a Besicovitch set and that
$S\subset\cB$ is finite.
\begin{compactenum}[a)]
\item $\epsilon_S=\sup\{d_1(\eta,\sigma^r\eta): r\in\ell_S\Z\}=
2\,d(\cM_\cB\setminus(\ell_S+\cM_\cB))$.
\item $\cN_{M\epsilon_S}(\cO(\tilde{\eta}))\le\ell_S$ for each $M>1$.
\end{compactenum}
\end{lemma}

\begin{proof}
The second equality in a) follows from \eqref{eq:distance-density}. Next
let $r\in\Z$. There are $j\in\{1,\dots,\ell_S\}$ and $k\in\Z$ such that $r=j+k\ell_S$. Hence, observing the $S$-invariance of $d_1$ and Proposition~\ref{prop:like-Rogers},
\begin{equation*}
\td_1(\tsigma^r\tilde\eta,\tsigma^j\tilde\eta)=d_1(\sigma^r\eta,\sigma^j\eta)=d_1(\sigma^{k\ell_S}\eta,\eta)\le
d_1(\sigma^{\ell_S}\eta,\eta)
=\epsilon_S.
\end{equation*}
This implies the first equality in a) and also $\cO(\tilde{\eta})\subset\bigcup_{j=1}^{\ell_S}B_{\td_1}(\tsigma^j\tilde{\eta},M\epsilon_S)$, i.e.~assertion b).
\end{proof}

\begin{remark}
For a Besicovitch set $\cB=\{b_1,b_2,\dots\}$ let $S_n:=\{b_1,\dots,b_n\}$. Then Lemma~\ref{lemma:upper-estimate} shows that $\ac(\cO(\eta))\le\limsup_{n\to\infty}\frac{\log\ell_{S_{n+1}}}{-\log d(\cM_\cB\setminus(\ell_{S_n}+\cM_\cB))}$, where $d(\cM_\cB\setminus(\ell_{S_n}+\cM_\cB))\le d(\cM_\cB\setminus \cM_{S_n})\to0$. 
For Toeplitz sequences $\eta$, for which $\ac(\cO(\eta))=\ac(\overline{\cO(\eta)})$ by Remark~\ref{remark:minimal-case},
this is conceptually sharper than the upper estimate from \cite[Thm.~1.6]{FGJ-2016}, which has the same form but with $d(\cM_{\{\gcd(b,\ell_{S_n})\ :\ b\in\cB\}}\setminus\cM_{S_n})$ instead of $d(\cM_\cB\setminus(\ell_{S_n}+\cM_\cB))$ in the denominator (see also \cite[Prop.~3.7]{Dymek-Kasjan-Keller-2021}), and this density tends to zero if and only if~$\eta$ is a regular Toeplitz sequence. However, in our Example~\ref{ex:Toeplitz-invariant} below, both approaches yield the same upper estimate.
\end{remark}

For $r\in\N$ and finite $S\subset\cB$ let $\underline{\epsilon}_r:=\min\{d_1(\eta,\sigma^j\eta): 1\le j\le r\}$ and $\underline{\epsilon}_S:=\underline{\epsilon}_{\ell_S}$.

\begin{lemma}[Lower estimates]\label{lemma:lower-estimate}
\begin{compactenum}[a)]
\item Let $r\in\N$. Then $\cN_{\underline{\epsilon}_r}(\cO(\eta))\ge r$. In particular, $\cN_{\underline{\epsilon}_S}(\cO(\eta))\ge \ell_S$.
\item Let $r,r'\in\N$. Then $\underline{\epsilon}_r\ge d_1(\eta,\sigma^{r'}\eta)$, provided $\lcm\{1,\ldots,r\}\mid r'$. If, in addition, $\cB$ is coprime, it is enough to assume that $r'$ is divisible by every $b\in\cB$ that divides  $\lcm\{1,\ldots,r\}$.
\item If $\cB$ is pairwise coprime and $S=\cB\cap\{1,\dots,N\}$, then $\underline{\epsilon}_S= \epsilon_{S}$.
\end{compactenum}
\end{lemma}

\begin{proof}
a)
\;No open(!) ball $B_{\td_1}(\tsigma^i\tilde\eta,\underline{\epsilon}_r)$ can contain more than one of the points $\tsigma^j{\tilde\eta}$ $(j=1,\dots,r)$.\\
b)
\;The general statement follows from Proposition \ref{prop:like-Rogers}, whereas for the coprime case we use also Corollary \ref{cor:bar}.\\
c)
\;This follows from Corollary~\ref{coro:Erdos-distances}.
\end{proof}

\subsection{Toeplitz examples}\label{subsec:Toeplitz-examples}

The class of all Toeplitz $\cB$-free shifts is very large and varied, and there is little hope to find a general method to determine their Besicovitch covering numbers. So we content ourselves with a class of examples studied in \cite{Dymek2017}, and later in \cite{Dymek-Kasjan-Keller-2021}.

\begin{example}[The basic Toeplitz example]
Let $\cC=\{c_i:i>0\}$ where $3\le c_1<c_2<\dots$ is a sequence of pairwise coprime, odd integers, and let $\cB=\{b_i=2^{r_i}c_i:i>0\}$ with integers $1< r_1< r_2<\dots$. Then $\eta=\eta_\cB$ is a regular Toeplitz sequence, so the (usual) orbit closure of $\eta$ is minimal, has topological entropy zero, is indeed an isomorphic extension of its maximal equicontinuous factor, and has a trivial automorphism group, see \cite{BKKL2015,Dymek2017,KKL2016,Dymek-Kasjan-Keller-2021} for these and related results. Here we determine the covering numbers $\cN_\epsilon(\cO(\tilde{\eta}))$.

Denote $I_n=\{1,\dots, n\}$ and $S_I=\{b_i:i\in I\}$, $C_I=\{c_i:i\in I\}$ for finite $I\subset\N=\{1,2,3,\dots\}$, and abbreviate  $S_n=S_{I_n}$ and $C_n=C_{I_n}$. Then $\ell_n=\lcm(S_n)=2^{r_n}\lcm(C_n)$.
Observe that
\begin{equation}\label{eq:M_n-def}
M_n:=\max_{1\le \jm\le n}\left(\frac{c_{\jm+1}}{c_\jm}-1\right)^{-1}\le c_n.
\end{equation}
We will show below that
\begin{equation}\label{eq:underline-eps-bound}
\epsilon_{\ell_n}
=
2\cdot {d(\cM_\cB\setminus\cM_{S_n})} 
\;\text{ and }\; \underline{\epsilon}_{\ell_n}\ge M_n^{-1}\epsilon_{\ell_n}\ge c_n^{-1}\epsilon_{\ell_n},
\end{equation}
so that, in view of Lemmas~\ref{lemma:upper-estimate} and~\ref{lemma:lower-estimate},
\begin{equation}\label{eq:Neps-bounds-Toeplitz}
\cN_{2\epsilon_{\ell_n}}(\cO(\eta))\le\ell_n\le \cN_{c_n^{-1}\epsilon_{\ell_n}}(\cO(\eta)).
\end{equation}
In the next example we will use these estimates to determine the amorphic complexity of such systems for some particular choices of the numbers $r_i$ and $c_i$.\footnote{\eqref{eq:underline-eps-bound} together with \eqref{eq:Neps-bounds-Toeplitz} shows that the  upper estimate of the amorphic complexity for regular Toeplitz sequences from \cite[Thm.~1.6]{FGJ-2016} is sharp for our class of examples. }

We turn to the proof of \eqref{eq:underline-eps-bound}.
Let $0<\ell\in\Z$, $\ell=2^r s$ for some odd $s\in\Z$, denote
$J=\{i\in\N:b_i\mid \ell\}=\{i\in\N: r_i\le r,\ c_i\mid s\}$
and $k=\max\{i\in\N: r_i\le r\}$.
Observe that $S_J=\Bl=\{b\in\cB: b\mid\ell\}$ and $J\subseteq I_k$, but $J$ need not be an interval.
Recall from Lemma~\ref{lem:sum} that
\begin{equation}\label{eq:basic-sum-for toeplitz}
d(\cM_\cB\cup(\ell+\cM_\cB))
=
\sum_{\emptyset\neq I\subseteq\N:|I|<\infty}(-1)^{|I|+1}\frac{T_\ell(S_I)}{\lcm(S_I)},
\end{equation}
with
\begin{equation*}
\begin{split}
T_\ell(S)
&=
\card\{K\subseteq S\setminus \Bl: \gcd(\lcm(K),\lcm(S\setminus K))\mid \ell\}\\
&=
\card\{K\subseteq S\setminus S_J: K\subseteq S_k\text{ or }S\setminus K\subseteq S_k\}\\
&=
\begin{cases}
\card\{K\subseteq S\setminus S_J\}&\text{ if }S\subseteq S_k\\
\card\{K\subseteq S\setminus S_J: K\subseteq S_k\}
+\card\{K\subseteq S\setminus S_J: S\setminus S_k\subseteq K\}&\text{ if }S{\not\subseteq} S_k
\end{cases}\\
&=
\begin{cases}
2^{|S\setminus S_J|}&\text{ if }S\subseteq S_k\\
2\cdot 2^{|S\cap S_k\setminus S_J|}&\text{ if }S{\not\subseteq} S_k,
\end{cases}
\end{split}
\end{equation*}
where the infinite sum has to be taken with some care, see Lemma~\ref{lem:sum}.

We split the sum in \eqref{eq:basic-sum-for toeplitz} into two parts, where we always understand that sets $I,I'$ etc. are finite sets:
\begin{equation*}
\begin{split}
&d(\cM_\cB\cup(\ell+\cM_\cB))\\
&=
\sum_{\emptyset\neq I\subseteq I_k}(-1)^{|I|+1}\frac{2^{|S_I\setminus S_J|}}{\lcm(S_I)}
\;+\;\sum_{I'\subseteq I_k}(-1)^{|I'|}\frac{2^{|S_{I'}\setminus S_J|}}{\lcm(C_{I'})} \sum_{\emptyset\neq I\subseteq \N\setminus I_k}(-1)^{|I|+1}\frac{2}{\lcm(S_I)}
\\
&=
\sum_{\emptyset\neq I\subseteq I_k}(-1)^{|I|+1}\frac{2^{|I\setminus J|}}{\lcm(S_I)}\\
&\hspace*{9mm}
+\;\sum_{I'\subseteq J}(-1)^{|I'|}\frac{1}{\lcm(C_{I'})}
\sum_{I''\subseteq I_k\setminus J}(-1)^{|I''|}\frac{2^{|I''|}}{\lcm(C_{I''})}
\sum_{\emptyset\neq I\subseteq \N\setminus I_k}(-1)^{|I|+1}\frac{2}{\lcm(S_I)}\\
&=
\sum_{\emptyset\neq I\subseteq I_k}(-1)^{|I|+1}\frac{2^{|I\setminus J|}}{\lcm(S_I)}
\;+\;
d(\cF_{C_J})\cdot
\prod_{i\in I_k\setminus J}\left(1-\frac{2}{c_i}\right)
\cdot 2\,d(\cM_{\cB\setminus S_k})
\end{split}
\end{equation*}
Similarly, with the same splitting of sums,
one obtains
\begin{equation*}
\begin{split}
d(\cM_{\cB})
&=
\sum_{\emptyset\neq I\subseteq\N}(-1)^{|I|+1}\frac{1}{\lcm(S_I)}\\
&=
\sum_{\emptyset\neq I\subseteq I_k}(-1)^{|I|+1}\frac{1}{\lcm(S_I)}
\;+\;d(\cF_{C_J})\cdot\prod_{i\in I_k\setminus J}\left(1-\frac{1}{c_i}\right)\cdot d(\cM_{\cB\setminus S_k})
\end{split}
\end{equation*}
Hence, in view of \eqref{eq:distance-density},
\begin{equation}\label{eq:eps-as-difference-example}
\begin{split}
\frac{1}{2}\epsilon_\ell
&=
\frac{1}{2}d_1(\eta,\sigma^\ell\eta)
=
d(\cM_\cB\cup(\ell+d(\cM_\cB))-d(\cM_\cB)\\
&=
\sum_{\emptyset\neq I\subseteq I_k}(-1)^{|I|+1}\frac{2^{|I\setminus J|}-1}{\lcm(S_I)}\\
&\qquad+
d(\cF_{C_J})\cdot\left(\prod_{i\in I_k\setminus J}\left(1-\frac{2}{c_i}\right)-\prod_{i\in I_k\setminus J}\left(1-\frac{1}{c_i}\right)\right)\cdot d(\cM_{\cB\setminus S_k})
\\
&\qquad+
d(\cF_{C_J})\cdot
\prod_{i\in I_k\setminus J}\left(1-\frac{2}{c_i}\right)
\cdot d(\cM_{\cB\setminus S_k})
\end{split}
\end{equation}

For $\ell=\ell_n$ we have $r=r_n$, $k=n$ and $J=I_n$, so that
\begin{equation}\label{eq:eps_ell_n}
\frac{1}{2}\epsilon_{\ell_n}=d(\cF_{C_n})\,d(\cM_{\cB\setminus S_n})
{\;=d(\cM_{\cB\setminus S_n}\cap\cF_{C_n})=d(\cM_\cB\setminus\cM_{S_n}),}
\end{equation}
{where the second equality holds because $\cB\setminus S_n$ and $C_n$ are coprime.}
This is the first assertion of \eqref{eq:underline-eps-bound}.

We turn to $\ell\in[1,\ell_n)$:
To that end let $\jm=1+\max\{j\in\N: I_{j}\subseteq J\}$ and observe that
\begin{compactitem}
\item $m\le n$,
\item $j< \jm$ and $I\subseteq I_{j}$\quad$\Rightarrow$\quad$|I\setminus J|=0$,
\item $\jm\in I\subseteq I_{\jm}$\quad$\Rightarrow$\quad$|I\setminus J|=1$.
\end{compactitem}
Now the second line in \eqref{eq:eps-as-difference-example} evaluates as
\begin{equation*}
\begin{split}
&\sum_{\emptyset\neq I\subseteq I_k}(-1)^{|I|+1}\frac{2^{|I\setminus J|}-1}{\lcm(S_I)}
\;+\;
\sum_{I\subseteq I_k}(-1)^{|I|}\frac{2^{|I\setminus J|}-1}{\lcm(C_I)}\cdot d(\cM_{\cB\setminus S_k})
\\
&\quad=
\sum_{j\in I_k}\;\sum_{\emptyset\neq I\subseteq I_k,\ \max I=j}(-1)^{|I|+1}\frac{2^{|I\setminus J|}-1}{\lcm(C_I)}\left(\frac{1}{2^{r_j}}-d(\cM_{\cB\setminus S_k})\right)\\
&\quad=
\left(\frac{1}{2^{r_{\jm}}}-d(\cM_{\cB\setminus S_k})\right)\;
\sum_{I'\subseteq I_{\jm-1}}(-1)^{|I'|+2}\frac{1}{c_{\jm}\lcm(C_{I'})}\\
&\qquad\qquad +
\sum_{j=\jm+1}^k\left(\frac{1}{2^{r_j}}-d(\cM_{\cB\setminus S_k})\right)\;
\sum_{I'\subseteq I_{j-1}}(-1)^{|I'|+2}\frac{2^{(|I'\cup\{j\})\setminus J|}-1}{c_j\lcm(C_{I'})}\\
&\quad=: E_1+E_2.
\end{split}
\end{equation*}
As $d(\cM_{\cB\setminus S_k})\le\sum_{i=k+1}^\infty\frac{1}{2^{r_i}c_i}\le \frac{1}{2^{r_k}c_{k+1}}\le\frac{1}{2^{r_j}c_{\jm}}$ for each $j=\jm,\ldots,k$,
\begin{equation*}
E_1
\ge
\frac{1}{2^{r_{\jm}}}\cdot\left(1-\frac{1}{c_{\jm}}\right)\cdot\frac{1}{c_{\jm}}\cdot\prod_{i=1}^{\jm-1}\left(1-\frac{1}{c_i}\right)
=
\frac{1}{2^{r_{\jm}}c_{\jm}}\cdot\prod_{i\in I_\jm}\left(1-\frac{1}{c_i}\right)
\end{equation*}
and
\begin{equation*}
\begin{split}
E_2
&=
\sum_{j=\jm+1}^k\frac{1}{c_j}\left(\frac{1}{2^{r_j}}-d(\cM_{\cB\setminus S_k})\right)\sum_{I'\subseteq I_{j-1}\cap J}(-1)^{|I'|}\frac{1}{\lcm(C_{I'})}\\
&\hspace*{60mm}\sum_{I''\subseteq I_{j-1}\setminus J}(-1)^{|I''|}\frac{2^{|I''|+1_{\N\setminus J}(j)}-1}{\lcm(C_{I''})}
\end{split}
\end{equation*}
where the double sum $\sum_{I'\subseteq I_{j-1}\cap J}\cdots\sum_{I''\subseteq I_{j-1}\setminus J}\cdots$ equals
\begin{equation*}
\prod_{i\in I_{j-1}\cap J}\left(1-\frac{1}{c_i}\right)
\prod_{i\in I_{j-1}\setminus J}\left(1-\frac{1}{c_i}\right)\left(2^{1_{\N\setminus J}(j)}\prod_{i\in I_{j-1}\setminus J}\left(1-\frac{1}{c_i-1}\right)-1\right).
\end{equation*}
Hence
\begin{equation*}
\begin{split}
E_2
&\ge
-\sum_{j=\jm+1}^k\frac{1}{c_j}\left(\frac{1}{2^{r_j}}-d(\cM_{\cB\setminus S_k})\right)
\cdot \prod_{i\in I_{j-1}}\left(1-\frac{1}{c_i}\right)\\
&\ge
-\frac{2}{c_{\jm+1}2^{r_{\jm+1}}}\prod_{i\in I_{\jm}}\left(1-\frac{1}{c_i}\right)
\end{split}
\end{equation*}
Therefore, in view of \eqref{eq:eps-as-difference-example} and \eqref{eq:eps_ell_n} and observing that $I_m\subseteq I_n$ and
$d(\cM_{\cB\setminus S_n})\le 1/(2^{r_{n}}c_{n+1})$,
\begin{equation*}
\begin{split}
\frac{1}{2}\epsilon_\ell
&\ge
E_1+E_2
\ge
\frac{1}{2^{r_{\jm}}}\prod_{i=1}^{\jm}\left(1-\frac{1}{c_i}\right)\cdot\left(\frac{1}{c_{\jm}}-\frac{1}{c_{\jm+1}}\right)\\
&\ge
\frac{1}{2}\epsilon_{\ell_n}\cdot c_{n+1}\cdot\left(\frac{1}{c_{\jm}}-\frac{1}{c_{\jm+1}}\right)
\ge
\frac{1}{2}\epsilon_{\ell_n}\cdot \left(\frac{c_{\jm+1}}{c_\jm}-1\right).
\end{split}
\end{equation*}
This implies the second assertion of \eqref{eq:underline-eps-bound}.
\end{example}

\begin{example}[Non block code equivalent Toeplitz examples]\label{ex:Toeplitz-invariant}
For certain choices of numbers $r_i$ and $c_i$ from the previous example
we discuss how the \emph{dimensional scale} describes the dependence of $\cN_{\epsilon_{\ell_n}}(\cO(\eta))$ on $\epsilon_{\ell_n}$, where
$\ell_n=2^{r_n}c_1\cdots c_n$ and
$\epsilon_{\ell_n}=2 d(\cF_{C_n})\, d(\cM_{\cB\setminus S_n})$.
\\[2mm]
a)\; Consider the case $r_n=n$.
If $c_n^{-1}\epsilon_{\ell_n}\le \epsilon\le c_{n-1}^{-1}\epsilon_{\ell_{n-1}}$, then $\cN_\epsilon(\cO(\eta))\ge \cN_{c_{n-1}^{-1}\epsilon_{\ell_{n-1}}}(\cO(\eta))\ge\ell_{n-1}$, see~\eqref{eq:Neps-bounds-Toeplitz}.
Observing also that $d(\cF_{C_n})=\prod_{i=1}^n\left(1-\frac{1}{c_i}\right)\ge \frac{1}{n}$ for all $n\ge 2$ (recall that $c_i\ge 2i+1$) and $d(\cM_{\cB\setminus S_n})\ge 1/(2^{n+1}c_{n+1})$, this yields for all $\alpha>0$
\begin{equation*}
\cN_\epsilon(\cO(\eta))\cdot\epsilon^\alpha
\ge
\ell_{n-1}c_{n}^{-\alpha}\epsilon_{\ell_n}^\alpha
\ge
\const_\alpha \cdot n^{-\alpha}2^{n(1-\alpha)} c_1\cdots c_{n-1}(c_nc_{n+1})^{-\alpha}\to\infty
\end{equation*}
as $n\to\infty$, provided $c_{n+1}=O((c_1\cdots c_{n})^\gamma)$ for all $\gamma>0$, e.g. if $c_1<c_2<\dots$ are the odd prime numbers. In that case,
$\alpha=+\infty$ is not only the critical parameter for the growth of $\ell_n$ but also for the growth of $[\cN_\epsilon(\cO(\eta))]_\approx$ in the \emph{dimensional scale} $\epsilon\mapsto\epsilon^\alpha$, see again  \cite{Kloeckner2012} for details.
\\[2mm]
b)\; Suppose now that
$r_i=[\kappa\log_2(c_1\cdots c_i)]$ for some parameter $\kappa>0$
and, as before, $c_{n+1}=O((c_1\cdots c_{n})^\gamma)$ for each $\gamma>0$. Then $\ell_n\asymp (c_1\cdots c_n)^{1+\kappa}$ and $\epsilon_{\ell_n}\asymp\prod_{i=1}^n\left(1-\frac{1}{c_i}\right)\cdot(c_1\cdots c_{n+1})^{-\kappa}\cdot c_{n+1}^{-1}$.
Fix $q\in\N$ such that $q\kappa\ge 2$. There is some constant $K>0$ such that
\begin{equation*}
2\epsilon_{\ell_{n+q}}
<
K\cdot\epsilon_{\ell_n}\cdot c_n^{-q\kappa}(c_n^{-q}\cdot c_{n+2}\cdots c_{n+q+1})^{-\kappa}\cdot\frac{c_{n+1}}{c_{n+q}}
\le
Kc_n^{-1}\cdot c_n^{-1}\epsilon_{\ell_n}
\le
c_n^{-1}\epsilon_{\ell_n}
\end{equation*}
for all large $n$.
Therefore,
if $c_n^{-1}\epsilon_{\ell_n}\le \epsilon\le c_{n-1}^{-1}\epsilon_{\ell_{n-1}}$, then \eqref{eq:Neps-bounds-Toeplitz} implies
\begin{equation*}
\ell_{n-1}
\le
\cN_{c_{n-1}^{-1}\epsilon_{\ell_{n-1}}}(\cO(\eta))
\le
\cN_\epsilon(\cO(\eta))
\le
\cN_{c_n^{-1}\epsilon_{\ell_{n}}}(\cO(\eta))
\le
\cN_{2\epsilon_{\ell_{n+q}}}(\cO(\eta))
\le
\ell_{n+q},
\end{equation*}
and elementary computations show that
\begin{equation*}
\cN_\epsilon(\cO(\eta))\cdot\epsilon^\alpha
\to
\begin{cases}
0&\text{ if }\alpha\kappa>1+\kappa\\
+\infty&\text{ if }\alpha\kappa<1+\kappa
\end{cases}
\quad\text{ as }\epsilon\to0,
\end{equation*}
so that $\alpha=\frac{1+\kappa}{\kappa}$ is the critical parameter of $[\cN_\epsilon(\cO(\eta))]_\approx$ for the dimensional scale. With the terminology from \cite{FGJ-2016, FGJK-2021} this means that the $\cB$-free Toeplitz shift has amorphic complexity $\frac{1+\kappa}{\kappa}$.

We recall that all these examples give rise to a minimal entropy zero system, which is an isomorphic extension of its maximal equicontinuous factor (MEF), which in turn is isomorphic to its associated odometer group $G$. If we fix $c_1,c_2,\dots$, then all these odometer groups are isomorphic to $\Z_{2^\infty}\times \prod_{i\ge 1}\Z/c_i\Z$, see e.g.~\cite[Thm.~1.2]{Downarowicz2005}. Moreover, all these systems have a trivial automorphism group (i.e.~isomorphic to $\Z$), see~\cite{Dymek2017, Dymek-Kasjan-Keller-2021}. Nevertheless, as the above examples all have different critical dimension parameters $\alpha=\frac{1+\kappa}{\kappa}$ or $\alpha=\infty$, not any two of them can have $\approx$-equivalent Besicovitch covering numbers and hence not any two of them can be block code equivalent.
\end{example}

\subsection{The Erd\H{o}s case}\label{subsec:Erdos-examples}

Suppose now that all numbers in $\cB$ are pairwise coprime and that $\sum_{b\in\cB}\frac{1}{b}<\infty$. This is called the \emph{Erd\H{o}s case}.
Subsections~\ref{subsec:arithmetic} and~\ref{subsec:covering} contain a tool-box to estimate Besicovitch covering numbers of Erd\H{o}s examples. Here we just study the most classical one - the square-free numbers.

\begin{example}[Square-free numbers]\label{ex:square-free}
Let $\cB=\{p_i^2:i>0\}$ where $p_1<p_2<\dots$ is the sequence of all prime numbers, $S_n=\{p_1^2,\dots,p_n^2\}$, $\ell_n=\lcm(S_n)=p_1^2\cdots p_n^2=e^{2p_n\cdot(1+o(1))}$,
and let $\eta=\eta_\cB$. Then
\begin{equation*}
\epsilon_{\ell_n}=\frac{12}{\pi^2}\cdot\left(1-\prod_{i>n}\left(1-\frac{1}{p_i^2-1}\right)\right)
=
\frac{12}{\pi^2}\cdot \sum_{i>n}\frac{1}{p_i^2}\cdot(1+o(1))
=
\frac{12}{\pi^2}\cdot \frac{1+o(1)}{p_n\log p_n},
\end{equation*}
by Corollary~\ref{cor:formula-G}, and it follows from Lemmas~\ref{lemma:upper-estimate} and~\ref{lemma:lower-estimate} that
\begin{equation*}
\cN_{M\epsilon_{\ell_n}}(\cO(\eta))\le\ell_n\le \cN_{\epsilon_{\ell_n}}(\cO(\eta))\;
\text{ for each $M>1$.}
\end{equation*}
As, for each $\xi>0$,
\begin{equation*}
\begin{split}
\ell_n\cdot \exp\left(-(\xi\epsilon_{\ell_n})^{-\alpha}\right)
&=
\exp\left(2p_n\cdot(1+o(1))-(12\xi/\pi^2)^{-\alpha}\, (p_n\log p_n)^{\alpha}\cdot(1+o(1))\right)\\
&\to
\begin{cases}
0&\text{ if }\alpha>1\\
+\infty&\text{ if }\alpha<1,
\end{cases}
\end{split}
\end{equation*}
$\alpha=1$ is not only the critical parameter for the growth of $\ell_n$ but also for the growth of $\cN_\epsilon(\cO(\eta))$ on the \emph{power exponential scale} $\epsilon\mapsto\exp(-\epsilon^{-\alpha})$, see \cite{Kloeckner2012} for details. (Observe that  $\epsilon_{\ell_n}/\epsilon_{\ell_{n+1}}\to 1$.) Note also that in the \emph{polynomial (or dimensional) scale} the critical parameter is $+\infty$,
i.e.~$\eta$ has infinite amorphic complexity.
\end{example}

\begin{remark}
In the Erd\H{o}s case, Besicovitch covering numbers are not useful as invariants for block code equivalence, because two such systems are block code equivalent if and only
their underlying sets $\cB$ are identical.\footnote{Private communication by Mariusz Lema\'n{}czyk. The argument is inspired by \cite{Mentzen2017} and extends to the broader class of hereditary $\cB$-free shifts as in \cite{KLRS2022}.}
\end{remark}

\section{$0$-$1$-sequences generated by the golden rotation}\label{sec:rotation}

\subsection{Coding with irrational rotations}

Let $\alpha\in\R\setminus\Q$, $\T=\R/\Z$ and denote $R_\alpha:\T\to\T, h\mapsto h+\alpha \mod 1$. Consider a Borel set $W\subset\T$ with $0<\lambda(W)<1$, where $\lambda$ denotes normalized Haar measure on $\T$, i.e.~Lebesgue measure in this case.
The window $W$ determines a coding map
$\varphi:\T\to\{0,1\}^\Z$, $(\varphi(h))_k=1_W(R_\alpha^kh)$.
In Subsection~\ref{subsec:construction-W} we construct, for $\alpha$ being the inverse of the golden number,
a family of Borel measurable windows $W_{s}\subset\T$ $(s>1)$ with the following properties:
\begin{compactenum}[(P1)]
\item \label{item:P1}
For each $s>1$, the orbit of $\varphi_s(h)$ is dense in $\{0,1\}^\Z$ for $\lambda$-a.e.~$h$, where $\varphi_s$ is the coding map defined by the window $W^s$.
\item \label{item:P2}
All $W^s$ are Haar aperiodic.
\item \label{item:P3}
$\lambda\{h\in\T:\lambda(W^s\triangle(W^s+h))\le \epsilon\}{\ \approx\ } \epsilon^{s/(s-1)}$ as $\epsilon\to0$, i.e. there is $C>0$ such that $C^{-1}\epsilon^{s/(s-1)}\le \lambda\{h\in\T:\lambda(W^s\triangle(W^s+h))\le \epsilon\}\le C\epsilon^{s/(s-1)}$ for sufficiently small $\epsilon>0$.
\end{compactenum}

These properties allow to prove
\begin{proposition}\label{prop:rotation}
Let $s>1$ and denote $\mu_s:=\lambda\circ\varphi_s^{-1}$. Then
\begin{compactenum}[a)]
\item $(\{0,1\}^\Z,\sigma,\mu_s)$ has spectrum $\{e^{2\pi i \ell\alpha}:\ell\in \Z\}$,
\item $\overline{\cO(x)}=\{0,1\}^\Z$ for $\mu_s$-a.e.~$x$, and
\item the amorphic complexity of $\cO(x)$ equals $\frac{s}{s-1}$ for all $x\in X_\mu$.
\end{compactenum}
Hence, if $1<s<s'$, then $x$ and $x'$ are not block code equivalent for $\mu_s$-a.e.~$x$ and $\mu_{s'}$-a.e.~$x'$, although
they are generic for measures with identical spectra and
their common orbit closure $\{0,1\}^\Z$ has a vast (though countable) automorphism group (see e.g.\ \cite{Salo2018}).
\end{proposition}
\begin{proof}
a)\;The following fact is a special case of \cite[Th. B1']{KRS2023}:
If $W$ is Haar aperiodic, then the dynamical system $(\{0,1\}^\Z,\sigma,\lambda\circ\varphi^{-1})$ has spectrum
$\{e^{2\pi i \ell\alpha}:\ell\in \Z\}$.
So the assertion follows from~(P\ref{item:P2}).\\
b)\;follows from~(P\ref{item:P1}).\\
c)\;
For all $\epsilon>0$ and $x,y\in X_\mu$,
\begin{equation*}
\mu(B_{d_1}(x,\epsilon))^{-1}\le\cN_\epsilon(\cO(x))=\cN_\epsilon(\cO(y))\le \mu(B_{d_1}(x,\epsilon/2))^{-1}.
\end{equation*}
by Theorem~\ref{theo:precisions}c), so it
suffices to prove the claim for just one $x\in X_\mu$.
Recall that $G=\T$ and let $g\in G_0\cap\varphi^{-1}(X_\mu)\subseteq G$, see Lemma~\ref{lemma:d1-formula}.  Denote $x=\varphi(g)$. Then
\begin{equation*}
\begin{split}
\mu(B_{d_1}(x,\epsilon))&=\lambda\{g'\in\T: d_1(\varphi(g'),\varphi(g))<\epsilon\}\\
&=
\lambda\{g'\in\T:\lambda((W-g')\triangle(W-g))<\epsilon\}\\
&=
\lambda\{h\in\T:\lambda(W\triangle(W+h))<\epsilon\}.
\end{split}
\end{equation*}
Now assertion c) follows from property (P\ref{item:P3}).
\end{proof}

\subsection{The golden rotation}

Denote by $\phi=\frac{\sqrt{5}+1}{2}$ the golden number and observe the following facts:
\begin{compactenum}[(GR1)]
\item $\phi$ and $-\phi^{-1}=\frac{-\sqrt{5}+1}{2}$ are the roots of $z^2-z-1$, so $\phi^2=\phi+1$ and $\phi-1=\phi^{-1}$.
\item Let $\alpha:=\phi-1=\phi^{-1}$. Then $\alpha$ and $-\alpha^{-1}=-\phi$ are the roots of $z^2+z-1$, so $\alpha^2=1-\alpha$ and $\alpha+1=\alpha^{-1}$.
\item Denote by $q_0=1$, $q_1=1$ and $q_{n+1}=q_n+q_{n-1}$ $(n>0)$ the Fibonacci numbers. Then
\begin{equation*}
q_n=\frac{\phi^2}{\phi^2+1}\phi^{n}+\frac{1}{\phi^2+1}(-\phi)^{-n}.
\end{equation*}

\item Let $p_0=0$ and $p_n=q_{n-1}$ $(n>0)$, and denote $\theta_n:=(-1)^n(q_n\alpha-p_n)$. Then $\theta_0=\alpha$, $\theta_1=\alpha^2$, and $\theta_{n+1}=-\theta_n+\theta_{n-1}$, so that $\theta_n=\alpha^{n+1}$.
\item
Let $\T=\R/\Z$ and denote $R:\T\to\T,x\mapsto x+\alpha \mod 1$. For $k\in\Z$ let $c_k:=R^k(0)$ and observe that $c_{q_n}=q_n\alpha-p_n=(-1)^{n}\theta_n \mod 1$. For even $n$, $c_{q_{n-1}},c_{q_n}$ and $c_{q_{n+1}}$ are arranged as follows:\\
\includegraphics[scale=1,clip,trim=0 30 0 30]{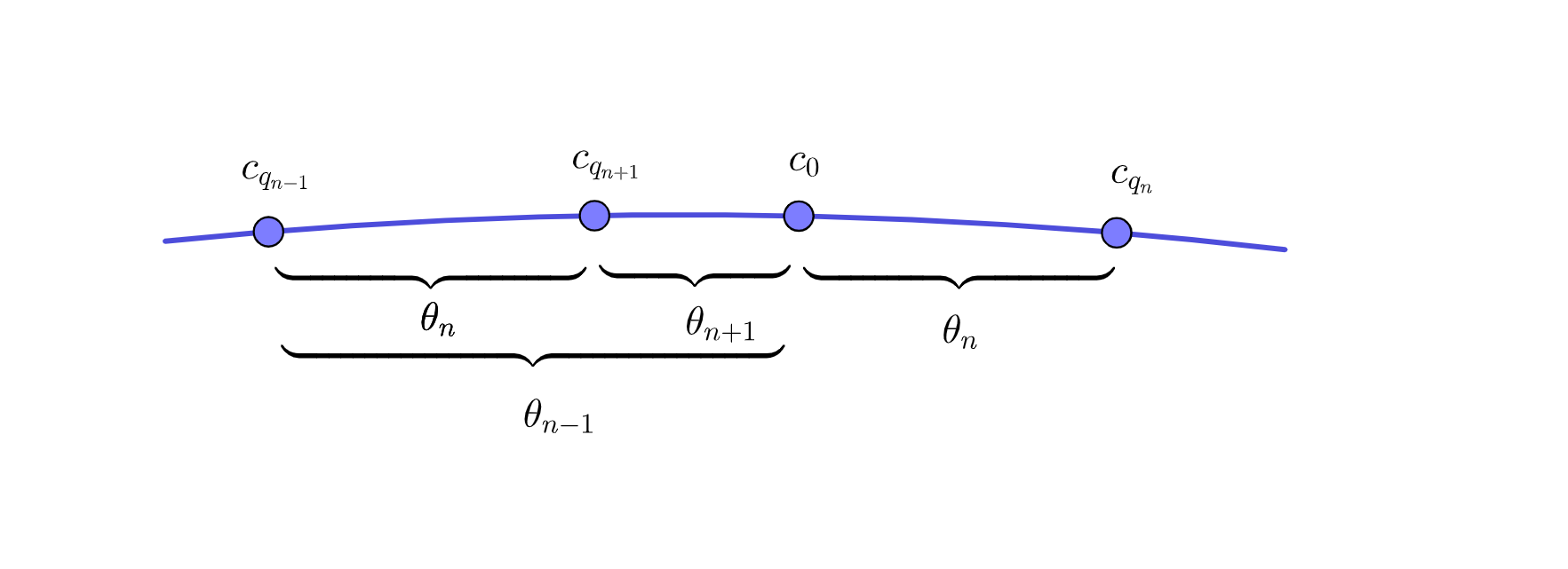}\\
For odd $n$, the orientation of points in this sketch is reversed.
Denote by $J_n$ the interval $(c_{q_{n-1}},c_{q_n}]\subseteq \T$ if $n$ is even and $[c_{q_{n}},c_{q_{n-1}})\subseteq \T$ if $n$ is odd.
\item\label{item:GR6}
If $c_k\in J_n$ for some $k>0$, then $k\ge q_n$, see \cite[line 5 from below on p.28]{dMvS} (and the discussion on the preceding pages).
\end{compactenum}

\subsection{Construction of the windows $W^s$}\label{subsec:construction-W}
In this section we fix a parameter $s>1$ and construct a Borel set $W^s\subset\T$ satisfying (P\ref{item:P1}) -- (P\ref{item:P3}). Observe that the validity of (P\ref{item:P1}) -- (P\ref{item:P3}) remains unchanged when $W^s$ is replaced by some ${W'}^s$ such that $\lambda(W^s\triangle {W'}^s)=0$.
In particular we can construct $W^s$ as a limit of sets $W^s_n$ in the sense that $\lambda(W^s\triangle W^s_n)\to0$ as $n\to\infty$.

Since $s$ is fixed, we skip it as an index wherever possible. So we write $W$ instead of $W^s$ and $W_n$ instead of $W_n^s$.
Moreover we write $\delta_n$ for $\nicefrac12\,\alpha^s\theta_{[sn]}$.
It may help to keep in mind that
\begin{equation*}\label{eq:delta-theta}
\nicefrac\alpha2\,\theta_n^s\le \nicefrac12\,\theta_n^s\alpha^{[sn]-sn+1}=\nicefrac12\,\alpha^{s+[sn]+1}=\nicefrac12\,\alpha^s\theta_{[sn]}=\delta_n\le\nicefrac12\,\theta_{n}^s.
\end{equation*}
Denote by $I_n^1$ the interval $(-\delta_n,0]\subset\T$, let $I_n^0=(0,\delta_n]$ and $I_n=(-\delta_n,\delta_n]$, so that $I_n^0\cup I_n^1= I_n$. Observe the following:
\begin{compactenum}[(O1)]
\item \label{item:returns}
Suppose that $k,k'\in\Z$, $0\le n\le n'$ and $(c_{k'}+I_{n'})\cap (c_k+I_n)=R^{k'}(I_{n'})\cap R^k(I_n)\neq\emptyset$. Then $(c_{k'-k}+I_{n'})\cap I_n\neq\emptyset$, so that $c_{k'-k}=-c_{k-k'}\in[-2\delta_n,2\delta_n]$.
Since $[-2\delta_n,2\delta_n]\subset J_{[sn]}$ by definition of $\delta_n$, we conclude that $|k'-k|\ge q_{[sn]}$.
\item \label{item:number_of_returns}
Denote by $T_{n,n'}$ the set of all
$k'\in[0,q_{n'+1})$ such that
$(c_{k'}+I_{n'})\cap I_n\neq\emptyset$. Then, in view of the preceding discussion, $c_{k'}\in [-2\delta_n,2\delta_n]$ for $k'\in T_{n,n'}$. Since the distance of $c_{k'}$ and $c_{k''}$ for two different $k',k''\in T_{n,n'}$ is the same as the distance of $c_{|k''-k'|}$ to zero, and as $|k''-k'|<q_{n'+1}$, it follows that $c_{|k''-k'|}\not\in J_{n'+1}$ by (GR\ref{item:GR6}), so the distance of $c_{|k''-k'|}$ to zero is at least $\theta_{n'+1}$. Therefore, $\card T_{n,n'}\le 4\delta_n/\theta_{n'+1}+1$.
\end{compactenum}

Since $[n/4]\cdot2^{[n/4]}\le q_{n-1}=q_{n+1}-q_n$ for all $n\ge 1$ (check the first few $n$ and observe that $2^{1/4}<\phi$),
there exists an $\omega\in\{0,1\}^\Z$ such that each segment $\omega_{[q_n,q_{n+1})}$ contains all $0$-$1$-blocks of length $[n/4]$ for all $n\ge1$.
\begin{equation*}
V_n^1=\bigcup_{q_{n}\le k<q_{n+1}}c_k+I_n^{\,\omega_{k-q_n}}\;\text{ and }\;
V_n^0=\bigcup_{q_{n}\le k<q_{n+1}}c_k+I_n^{\,1-\omega_{k-q_n}},
\end{equation*}
so that
\begin{equation*}
V_n^1\cup V_n^0=\bigcup_{q_{n}\le k<q_{n+1}}c_k+I_n.
\end{equation*}
Then $V_n^1\cap V_n^0=\emptyset$ and, because of~(O\ref{item:returns}), all three sets are disjoint unions of $q_{n-1}$ intervals of length $\delta_n$, respectively $2\delta_n$.

So fix some $\omega\in\{0,1\}^\Z$ as above. In particular, $\omega$ has a  dense shift orbit. The construction of the $W_n$ proceeds as follows:
\begin{equation*}
W_1=V_1^1=V_1^1\setminus V_1^0,
\end{equation*}
and for $n>1$
\begin{equation}\label{eq:W_n-def}
W_{n}=(W_{n-1}\setminus V_n^0)\cup V_n^1
=(W_{n-1}\cup V_n^1)\setminus V_n^0.
\end{equation}
We note some simple consequences:
\begin{compactitem}
\item $W_n\cap V_n^1=V_n^1$ and $W_n\cap V_n^0=\emptyset$.
\item Let $x\in I_n$. Then $R^kx\in c_k+I_n$ for all $k$.  If $q_n\le k<q_{n+1}$, then $R^kx\in V_n^1\cup V_n^0$, so that $R^kx\in W_n$ if and only if $R^kx\in V_n^1\cap(c_k+I_n)= c_k+I_n^{\,\omega_{k-q_n}}$, equivalently $x\in I_n^{\,\omega_{k-q_n}}$. Hence,
for all $n\ge 1$,
\begin{equation}\label{eq:W_n-equivalence}
\forall x\in I_n\; \forall k\in[q_n,q_{n+1}):\, R^kx\in W_n\,\Leftrightarrow\,\omega_{k-q_n}=1.
\end{equation}
\item $\lambda(W_{n-1}\triangle W_n)\le \lambda(V_n^0\cup V_n^1)=2q_{n-1}\delta_n\le \phi^{n-1}\theta_{[sn]}\le \phi^n\alpha^{sn}=\alpha^{(s-1)n}$.
\item It follows that there exists a Borel set $W$ such that
\begin{equation}\label{eq:W triangle Wn}
\lambda(W\triangle W_n)\le \epsilon_n:=\frac{\alpha^{(s-1)(n+1)}}{1-\alpha^{s-1}}\to0\;\text{ and }\;
W\triangle W_n\subseteq\bigcup_{n'>n}V_{n'}^1\cup V_{n'}^0.
\end{equation}
\end{compactitem}
So we need to show that \eqref{eq:W_n-equivalence} continues to hold for
sufficiently large $n$ and
``many'' $x\in I_n$, when $W_n$ is replaced by $W$.\footnote{\label{foot:arbitrary-choice}We defined $I_n^\pm$ such that $0\in I_n^+$ and $0\not\in I_n^-$. We could just as well have made an arbitrary choice of this kind for each pair of intervals $c_k+I_n^{\,\omega_{k-q_n}}$ and $c_k+I_n^{\,1-\omega_{k-q_n}}$ separately. Since it follows from (O\ref{item:returns}) that no $c_k$ is contained in any interval $c_{k'}+I_{n'}$ with $k'\in[q_{n'},q_{n'+1})$, such choices would allow to produce an arbitrary element from $\{0,1\}^\Z$ as $\varphi_s(0)$, so that  $\varphi_s(0)\not\in X_\mu$ in general. See also Remark~\ref{remark:G_0=G?}.}

To that end let $n\ge 1$, $x\in I_n$ and $k\in\Z$, and suppose that $R^kx\in W\triangle W_n$ with $k\in[q_n,q_{n+1})$. Then there is $n'>n$ such that $R^kx\in V_{n'}^1\cup V_{n'}^0=\bigcup_{q_{n'}\le k'<q_{n'+1}}c_{k'}+I_{n'}$.
Hence $R^kx\in (c_k+I_n)\cap(c_{k'}+I_{n'})$ for some $k'\in[q_{n'},q_{n'+1})$,
so that $|k'-k|\ge q_{[sn]}$ in view of~(O\ref{item:returns}). Hence $q_{n'+1}>k'\ge q_n+ q_{[sn]}> q_{[sn]}$, so that $n'\ge [sn]$.

If $(c_k+I_n)\cap(c_{k'}+I_{n'})\neq\emptyset$ for some
$k\in[0,q_{n'})$ and
$k'\in[q_{n'},q_{n'+1})$, then $k'-k\in T_{n,n'}$, the index set defined in~(O\ref{item:number_of_returns}).
Hence $R^kx\in \bigcup_{k'\in k+T_{n,n'}}(c_{k'}+I_{n'})$, i.e.
$x\in \bigcup_{k''\in T_{n,n'}}(c_{k''}+I_{n'})$. Therefore
\begin{equation}\label{eq:local-exponential estimate-1}
\begin{split}
&\lambda\left\{x\in I_n: \exists k\in[q_n,q_{n+1})\text{ s.t. }R^kx\in W\triangle W_n\right\}\\
&\le
\sum_{n'\ge [sn]}\lambda\left(\bigcup_{k''\in T_{n,n'}}(c_{k''}+I_{n'})\right)\\
&\le
\sum_{n'\ge [sn]}\card T_{n,n'}\cdot\lambda(I_{n'})
\le
\sum_{n'\ge [sn]}(4\delta_n/\theta_{n'+1}+1)\cdot 2\delta_{n'}\\
&\le 
\ \const\cdot\lambda(I_n)\cdot\sum_{n'\ge [sn]}\left(\alpha^{(s-1)n'}+\alpha^{s(n'-n)}\right)
\le \const\cdot\lambda(I_n)\cdot \alpha^{(s-1)sn}
\end{split}
\end{equation}
with constants that depend on $s$.\footnote{The $s$-dependence of the second constant is of the order $1/(1-\alpha^{s-1})$.}
So fix $n\ge 1$ large enough that $\const\cdot\lambda(I_n)\cdot \alpha^{(s-1)sn}\le\lambda(I_n)/2$. Then, by the ergodicity of the irrational rotation $R$,
for $\lambda$-a.e.~$x$ there is some $t>0$ such that $(\varphi(x))_{t+k}=\omega_{k}$ for all $k\in[q_n,q_{n+1})$. It follows that, except for $x$ in an exceptional set of $\lambda$-measure zero,
all $\varphi(x)$ have a dense shift orbit in $\{0,1\}^\Z$.
 This is (P\ref{item:P1}).

We turn to (P\ref{item:P3}). Let $H_{n}=\{h\in\T:\lambda(W_n\triangle(W_n+h))\le 3\epsilon_n\}$.
We will show that there is some integer $L>0$ (which depends on $s$) such that for all $n\ge L$
\begin{eqnarray}
\hspace*{7mm}\{h\in\T:\lambda(W\triangle(W+h))<\epsilon_{n+L}\}
\subseteq&\hspace*{-3mm} H_{n} \subseteq&\hspace*{-3mm}
\{h\in\T:\lambda(W\triangle(W+h))<\epsilon_{n-L}\}.\label{eq:two-inclusions-1}\\
I_{n+L}\subseteq&\hspace*{-3mm} H_{n}\subseteq&\hspace*{-3mm} I_{n-L},\label{eq:two-inclusions-2}
\end{eqnarray}
Then, for $n\ge 2L$ and $\epsilon\in(\epsilon_{n+1},\epsilon_{n}]$
\begin{equation}\label{eq:final-inclusions}
{I_{n+2L+1}}\;\subseteq\;\{h\in\T:\lambda(W\triangle(W+h))<\epsilon\}\;\subseteq\;I_{n-2L},
\end{equation}
so that there is a constant $C>0$ (which depends on $s$) such that
\begin{equation*}
\begin{split}
C^{-1}\epsilon^{s/(s-1)}
\le
\alpha^{{s+s(n+2L+1)}}
&\le \lambda\{h\in\T:\lambda(W\triangle(W+h))<\epsilon\}\\
&\le \alpha^{s(n-2L)}
\le C\epsilon^{s/(s-1)}.
\end{split}
\end{equation*}
This proves (P\ref{item:P3}), once the inclusions \eqref{eq:two-inclusions-1} and \eqref{eq:two-inclusions-2} are verified. For \eqref{eq:two-inclusions-1} it suffices observe that, in view of~\eqref{eq:W triangle Wn}, $|\lambda(W\triangle (W+h))-\lambda(W_n\triangle (W_n+h))|\le 2\lambda(W\triangle W_n)\le 2\epsilon_n$ and to choose $L$ so large that $\frac{\epsilon_{n-L}}{\epsilon_n}=\alpha^{-(s-1)L}>5$. {The l.h.s. inclusion follows since $\epsilon_{n+L}+2\epsilon_n\le 3\epsilon_n$.}
We turn to \eqref{eq:two-inclusions-2}: Recall that all $W_n$ are finite unions of intervals and denote the numbers of disjoint intervals by $w_n$. Then, by construction in \eqref{eq:W_n-def},
$w_n\le w_{n-1}+2q_{n-1}$ {(the term $W_{n-1}\setminus V_n^0$ produces at most $w_{n-1}+q_{n-1}$ intervals)}, and since $w_N=q_{N-1}\le 4q_N$,
we get inductively $w_n\le 4q_{n}$ (because $q_{n-1}/q_n\le 2/3$ for all $n$). Hence $\lambda(W_n\triangle (W_n+h))\le 4q_n|h|\le 4q_n2\delta_{n+L}$ for $h\in I_{n+L}$. Since $q_n\delta_{n+L}\le\const_s\cdot \phi^n\alpha^{s(n+L)}=\const_s\cdot\alpha^{sL}\alpha^{(s-1)n}=\const_s\cdot\alpha^{sL}\epsilon_{n}$, one can choose $L$ so large that the l.h.s.~inclusion holds.
For the second inclusion note that, by construction,  the indicator function $1_{W_n}$ has a jump at all $c_k$ $(q_n\le k<q_{n+1})$ and is constant to both sides of $c_k$ on intervals of length at least $\delta_n$.
Hence $\lambda(W_n\triangle(W_n+h))\ge |h|q_{n-1}$ for $h\in I_n$, so that
$H_n\subseteq\{h\in\T:|h|\le 3\epsilon_n/q_{n-1}\}$, and the latter set is easily seen to be contained in $I_{n-L}$ {for every $n\ge 2N+L$} when $L$ is sufficiently large.

Finally, (P\ref{item:P2}) follows from \eqref{eq:final-inclusions}, because the latter implies $h\in\bigcap_{n\ge 2L}I_{n-2L}=\{0\}$ whenever
${\lambda}(W\triangle(W+h))=0$.

\section{Some proofs}\label{sec:proofs}

\subsection{On densities related to Besicovitch sets}\label{subsec:on-densities}

For a set $A\subseteq\Z$ we denote by $d_{\N}(A)$, resp. $d_{-\N}(A)$ the density of $\N\cap A$ (resp. $(-\N)\cap A$) as a subset of $\N$ (resp. $-\N$), if it exits. the Symbols like $\dl_{\N}$, $\du_{\N}$ etc. have analogous meaning. Observe that if for a set $A\subseteq \Z$ both densities $d_{\N}(A)$, $d_{-\N}(A)$ exist and $d_{\N}(A)=d_{-\N}(A)$, then
\begin{equation*}
d_{\N}(A)=d_{-\N}(A)=\lim\limits_{n\rightarrow+\infty}\frac{1}{2n+1}|[-n,n]\cap A|=:d(A).
\end{equation*}
We claim that
if $\cB$ is Besicovitch, then
\begin{equation}\label{eq:trzy}
\begin{split}
d_{\N}(\cM_{\cB}\cup(\cM_{\cB}+r))=d_{-\N}(\cM_{\cB}\cup(\cM_{\cB}+r))=d(\cM_{\cB}\cup(\cM_B+r)).
\end{split}
\end{equation}
First observe that this is true if $\cB$ is finite (since in that case $\cM_{\cB}\cup(\cM_{\cB}+r)$ is a union of finitely many arithmetic progressions) and true also in the case $r=0$ and $\cB$ Besicovitch (then it follows since $\cM_{\cB}=-\cM_{\cB}$).
Observe that for every $A\subseteq\Z$ and every $r\in\Z$ we have
\begin{equation*}
\du_{\N}(A\cup (A+r))\le 2\du_{\N}(A).
\end{equation*}
Fix $\varepsilon>0$. Since $d(\cM_{\cB})$ exists, it follows by Davenport-Erd\"os theorem that for some finite  $S\subset\cB$ we have $\du(\cM_{\cB}\setminus\cM_S)<\varepsilon$. Then
$\du_{\N}((\cM_{\cB}\cup(\cM_{\cB}+r))\setminus (\cM_S\cup(\cM_S+r)))<2\varepsilon$ so that
\begin{equation*}
\du_{\N}(\cM_{\cB}\cup(\cM_{\cB}+r))
\le
d_{\N}(\cM_S\cup(\cM_S+r))+2\varepsilon
\le
\dl_{\N}(\cM_{\cB}\cup(\cM_{\cB}+r))+2\varepsilon.
\end{equation*}
It follows that the density $\du_{\N}(\cM_{\cB}\cup(\cM_{\cB}+r))$ exists and it is equal to the limit of $d_{\N}(\cM_S\cup(\cM_S+r))$ as $S\rightarrow \cB$. Analogous statement is true for $d_{-\N}$. Now, since (\ref{eq:trzy}) holds for finite sets $\cB$, the claim follows.
We note a simple consequence:
\begin{corollary}\label{coro:density-exists}
If $\cB$ is a Besicovitch set, then the following density exists:
$$
d(\cM_\cB\triangle(\cM_\cB+r))=2\left(d(\cM_\cB\cup(r+\cM_\cB))-d(\cM_\cB)\right)
$$
\end{corollary}
\begin{proof}
As the densities of $\cM_\cB$, $r+\cM_\cB$, and $\cM_\cB\cup(r+\cM_\cB)$ exist, also the densities of $(\cM_\cB\cup(r+\cM_\cB))\setminus\cM_\cB$ and $(\cM_\cB\cup(r+\cM_\cB))\setminus(r+\cM_\cB)$ exist, and as $\cM_\cB\triangle(r+\cM_\cB)$ is the disjoint union of the two latter sets, also its density exists.
\end{proof}

\subsection{Proof of Proposition~\ref{prop:like-Rogers}}\label{subsec:Proof-of-Prop-Rogers}
We adapt the proof of Rogers' Theorem (in \cite[Ch.~V, \S 3.]{Halberstam-Roth}) to prove the following refinement.
\begin{theorem}\label{th:rogers}
Consider  a sequence $(b_i)_{i\in I}$  of natural numbers and a sequence $(g_i)_{i\in I}$  of integers indexed by a countable  set $I$. Assume moreover that there is a finite decomposition $I=I_1\cup\ldots\cup I_N$ and integers  $\overline{g}_j$ for $j=1,\ldots,N$ such that $g_i=\overline{g}_j$ when $i\in I_j$.
Then, for every $m\in\Z$:
\begin{equation}\label{eq:rogers}
\delta(\bigcup_{i\in I}(g_i+b_i\Z))\ge \delta(\bigcup_{i\in I}(mg_i+b_i\Z)),
\end{equation}
in particular the logarithmic densities on both sides exist.
\end{theorem}

\begin{proof}
Let us first consider the case of finite $I$. In such case we can assume that $I=\{1,\ldots,N\}$, $I_j=\{j\}$ and $\overline{g}_j=g_j$ for $j\in I$.
We proceed by induction on the number $\omega(\ell)$ of the prime divisors of $\ell=\lcm\{b_i:i\in I\}$.
First observe that the assertion is trivially satisfied provided $b_i=1$ for some $i$.

1) If $\omega(\ell)=1$, then there exists a prime $p$ such that the numbers $b_i$ are powers of $p$.  In this case
$$
(g_i+b_i\Z)\cap (g_j+b_j\Z)\neq \emptyset,
$$
   if and only if one of the sets $g_i+b_i\Z, g_j+b_j\Z$ is contained in the other. Let $\preceq_g$ be the partial order in $I$ defined by $i\preceq_gj$ if $g_i+b_i\Z\subseteq g_j+b_j\Z$. Let $\max(\preceq_g)$ be the set of the maximal elements in $I$ with respect to this order. Then
   $$
   \bigcup_{i\in I}(g_i+b_i\Z)=\bigcup_{i\in \max(\preceq_g)}(g_i+b_i\Z).
   $$
Since the progressions $g_i+b_i\Z$ and  $g_j+b_j\Z$ are disjoint for incomparable $i,j$, it follows that
$$
\delta(\bigcup_{i\in I}(g_i+b_i\Z))=\sum_{i\in\max(\preceq_g)}\frac{1}{b_i}.
$$
Similarly,
$$
\delta(\bigcup_{i\in I}(mg_i+b_i\Z))=\sum_{i\in\max(\preceq_{mg})}\frac{1}{b_i},
$$
where $i\preceq_{mg}j$ if $mg_i+b_i\Z\subseteq mg_j+b_j\Z$. Observe that $i\preceq_gj$ implies $i\preceq_{mg}j$, so $\max(\preceq_{mg})\subseteq \max(\preceq_{g})$ and
\eqref{eq:rogers} follows.

2) Assume now that $N$ is a natural number and \eqref{eq:rogers} is true whenever $\omega(\ell)< N$. Assume that $\omega(\ell)=N$ and let
$$
\ell=u\cdot v,
$$
where $u$ and $v$ are natural numbers with $\omega(u)<N$, $\omega(v)<N$ and $\gcd(u,v)=1$.

We shall use the following notation: for $q\in\N$ we denote by $\Z_q$  the cyclic group $\Z/q\Z$ and given an integer $n$ we denote by $[n]_q$ the residue class $n+q\Z\in\Z_q$. Moreover, if $h$ is an element of a group $H$, then $\langle h\rangle$ denotes the subgroup of $H$ generated by $h$.  Clearly
\begin{equation}\label{eq:interpret}
\delta(\bigcup_{i\in I}(n_i+b_i\Z))=\frac{1}{\ell}\left|\bigcup_{i\in I}([n_i]_{\ell}+\langle [b_i]_{\ell}\rangle)\right|
\end{equation}
for every choice of $n_i\in\Z$ for $i\in I$.

Under our assumptions, thanks to CRT, the map $[n]_{\ell}\mapsto ([n]_u,[n]_v)$ establishes  an isomorphism of abelian groups
$$
\Pi:\Z_{\ell}\rightarrow \Z_u\times \Z_v,
$$
and this isomorphism maps a coset $[n]_{\ell}+\langle [b]_{\ell}\rangle$ of the subgroup generated by $[b]_{\ell}$ onto the product
$$
([n]_{u}+\langle [b]_{u}\rangle)\times ([n]_{v}+\langle [b]_{v}\rangle)
$$
for every $n\in\Z$ and $b\in\N$.
Let $m'$, $m''$ be integers such that $
m'=m \mod u$, $ m'=1 \mod v$, $m''=1 \mod u$, $m''=m \mod v$.
Then $m'm''=m\mod\ell$.

First we prove that  $\delta(\bigcup_{i\in I}(g_i+b_i\Z))\ge \delta(\bigcup_{i\in I}(m'g_i+b_i\Z))$, equivalently
\begin{equation}\label{eq:mprim}
\left|\bigcup_{i\in I}([g_i]_{\ell}+\langle [b_i]_{\ell}\rangle)\right|\ge \left|\bigcup_{i\in I}([m'g_i]_{\ell}+\langle [b_i]_{\ell}\rangle)\right|.
\end{equation}

Fix an element $a\in\Z_v$. Let $I_a=\{i\in I: a\in [g_i]_v+\langle [b_i]_{v}\rangle\}$. Then, as $m'=1\mod v$,  $I_a=\{i\in I: a\in [m'g_i]_v+\langle [b_i]_{v}\rangle\}$. Observe that
\begin{equation*}
(\Z_u\times \{a\})\cap \Pi\left(\bigcup_{i\in I}([g_i]_{\ell}+\langle [b_i]_{\ell}\rangle)\right)=\bigcup_{i\in I_a}([g_i]_{u}+\langle [b_i]_{u}\rangle)\times \{a\} .
\end{equation*}
Similarly
\begin{equation*}
(\Z_u\times \{a\})\cap \Pi\left(\bigcup_{i\in I}([m'g_i]_{\ell}+\langle [b_i]_{\ell}\rangle)\right)=\bigcup_{i\in I_a}([m'g_i]_{u}+\langle [b_i]_{u}\rangle)\times \{a\}.
\end{equation*}
Since $\omega(u)<\omega(\ell)$, by the inductive hypothesis we have
\begin{equation*}
\left|\bigcup_{i\in I_a}([g_i]_{u}+\langle [b_i]_{u}\rangle)\right|\ge \left|\bigcup_{i\in I_a}([m'g_i]_{u}+\langle [b_i]_{u}\rangle)\right|.
\end{equation*}
Since $a$ is arbitrary,  \eqref{eq:mprim} follows.

Similarly, interchanging the roles of $u$, $v$, we prove
\begin{equation}\label{eq:mbis}
\left|\bigcup_{i\in I}([m'g_i]_{\ell}+\langle [b_i]_{\ell}\rangle)\right|\ge \left|\bigcup_{i\in I}([m''m'g_i]_{\ell}+\langle [b_i]_{\ell}\rangle)\right|.
\end{equation}

As $m'm''=m\mod \ell$, \eqref{eq:mprim} and \eqref{eq:mbis} together with \eqref{eq:interpret} prove \eqref{eq:rogers}.

Now let us consider the general case.
It is enough to prove that the logarithmic  density of the set $\bigcup_{i\in F}(g_i+b_i\Z)$ \footnote{This set possesses logarithmic density, because it is a union of finitely many arithmetic progressions.}, where $F$ is a finite subset of $I$, converges (as $F$ converges to $I$) to the lower and to the upper logarithmic density of $\bigcup_{i\in I}(g_i+b_i\Z)$. The analogous statement with $g_i$ interchanged with $mg_i$ will follow in the same way. Observe that
$$
\bigcup_{i\in I}(g_i+b_i\Z)=\bigcup_{j=1}^N\bigcup_{i\in I_j}(\overline{g}_j+b_i\Z)=\bigcup_{j=1}^N\bigcup_{i\in F_j}(\overline{g}_j+b_i\Z)\cup\bigcup_{j=1}^ND_j,
$$
where $F_j=I_j\cap F$, for some sets $D_j\subseteq \bigcup_{i\in I_j}(\overline{g}_j+b_i\Z)$, $j=1,\ldots,N$ such that the sets
$$
\bigcup_{j=1}^N\bigcup_{i\in F_j}(\overline{g}_j+b_i\Z),D_1,\ldots,D_N
$$
are pairwise disjoint\footnote{We set $D'_j=\bigcup_{i\in I_j}(\overline{g}_j+b_i\Z)\setminus \bigcup_{i\in F_j}(\overline{g}_j+b_i\Z)$, and $D_1=D'_1$, $D_2=D'_2\setminus D_1$, $D_3=D'_3\setminus (D_1,\cup D_2)$.  etc.}
Then
$$
\overline{\delta}(\bigcup_{j=1}^N\bigcup_{i\in I_j}(\overline{g}_j+b_i\Z))\le \overline{\delta}(\bigcup_{j=1}^N\bigcup_{i\in F_j}(\overline{g}_j+b_i\Z))+\sum_{j=1}^N\overline{\delta}(D_j).
$$
Since $D_j\subseteq \bigcup_{i\in I_j}(\overline{g}_j+b_i\Z)\setminus \bigcup_{i\in F_j}(\overline{g}_j+b_i\Z)$ for $j=1,\ldots,N$, by Davenport-Erd\"os theorem, the sum
$\sum_{j=1}^N\overline{\delta}(D_j)$ can be made  smaller than fixed $\varepsilon>0$ provided $F$ contains  some finite set $F_{\varepsilon}\subseteq I$ depending on $\varepsilon$ only.
Therefore
$$
\overline{\delta}(\bigcup_{j=1}^N\bigcup_{i\in I_j}(\overline{g}_j+b_i\Z))-\varepsilon\le {\delta}(\bigcup_{j=1}^N\bigcup_{i\in F_j}(\overline{g}_j+b_i\Z))\le \underline{\delta}(\bigcup_{j=1}^N\bigcup_{i\in I_j}(\overline{g}_j+b_i\Z))
$$
if $F_{\varepsilon}\subseteq F$. The proof is complete.
\end{proof}

Theorem~\ref{th:rogers} yields the following
\begin{corollary}\label{cor:monotone} If $r|r'$, then
\begin{equation*}
\delta(\cM_{\cB}\cup(\cM_{\cB}+r))\ge \delta(\cM_{\cB}\cup(\cM_{\cB}+r')).
\end{equation*}
\end{corollary}

\begin{proof} We apply Theorem \ref{th:rogers} to $I=\cB\times\{0,1\}$, $b_{(c,j)}=c$, $g_{(c,0)}=0$, $g_{(c,1)}=r$ for $c\in\cB$, $j\in\{0,1\}$ and $m=\frac{r'}{r}$.
\end{proof}

\begin{proof}[Proof of Proposition~\ref{prop:like-Rogers}]
In the Besicovitch case also the usual densities of the sets from Corollary~\ref{cor:monotone} exist (see Subsection~\ref{subsec:on-densities}). Since they coincide with the logarithmic densities, this proves Proposition~\ref{prop:like-Rogers}.
\end{proof}

\bibliography{amorphic} \bibliographystyle{abbrv}

\end{document}